\documentclass[11pt]{article} \topmargin -0.25cm \oddsidemargin 0.5cm
\usepackage{amsthm}
\usepackage{amsmath}
\usepackage{amssymb}
\usepackage{enumerate}

\newcommand{\R}{\ensuremath{\mathbb R}}
\newcommand{\Rp}{\ensuremath{\R^+}}

\newcommand{\F}{\ensuremath{{\cal F}}}

\newcommand{\B}{\ensuremath{{\cal B}}}
\newcommand{\Ss}{{\mathcal S}}
\newcommand{\N}{\ensuremath{{\mathbb N}}}
\newcommand{\Ll}{\ensuremath{{\mathbb L}}}

\newcommand{\No}{\ensuremath{{\mathbb N}\cup\{0\}}}

\newcommand{\FF}{\ensuremath{{\mathcal F}}}

\newcommand{\arrowd}{\mathop{\rightarrow}_{\cal D}}
\newcommand{\arrowp}{\mathop{\rightarrow}_{\cal P}}
\newcommand{\Ppw}{P\text{-a.s.}}

\newcommand{\Ddet}{\mathcal{{D}}}

\newcommand{\D}{\ensuremath{{D}(\Rp,\R)}}
\newcommand{\DD}{\ensuremath{{D}(\Rp,\R^{2})}}
\newcommand{\DDD}{\ensuremath{{D}(\Rp,\R^{3})}}
\newcommand{\DDDD}{\ensuremath{{D}(\Rp,\R^{4})}}
\newcommand{\DDDDD}{\ensuremath{{D}(\Rp,\R^{5})}}
\newcommand{\Dd}{\ensuremath{{D}(\Rp,\R^{d})}}

\newcommand{\Ddddd}{\ensuremath{{D}(\Rp,\R^{4d})}}
\newcommand{\Dddddd}{\ensuremath{{D}(\Rp,\R^{5d})}}
\newcommand{\Ddddddd}{\ensuremath{{D}(\Rp,\R^{6})}}
\newcommand{\Ddddddddd}{\ensuremath{{D}(\Rp,\R^{8})}}
\newcommand{\Dddddddd}{\ensuremath{{D}(\Rp,\R^{7})}}

\newcommand{\lra}{\longrightarrow}

\newtheorem{theorem}{\bf Theorem}[section]
\newtheorem{proposition}[theorem]{\bf Proposition}
\newtheorem{lemma}[theorem]{\bf Lemma}
\newtheorem{corollary}[theorem]{\bf Corollary}

\newcommand{\Mh}{\mu_h}
\newcommand{\Ch}{c_h^2}
\newcommand{\ch}{c_h^1}
\newcommand{\Lh}{\lambda_h}

\newcommand{\SPl}[1]{{\mathbb{SP}}_l(h,#1)}

\newcommand{\SPu}[1]{\mathbb{SP}^u(h,#1)}

\newcommand{\SP}[1]{\mathbb{SP}_l^u(h,#1)}
\newcommand{\SPp}[3]{\mathbb{SP}_{#2}^{#3}(h,#1)}

\theoremstyle{definition}
\newtheorem{definition}[theorem]{\bf Definition}
\newtheorem{example}[theorem]{\bf Example}
\newtheorem{remark}[theorem]{\bf Remark}
\numberwithin{equation}{section}

\begin{document}
\title{\bf Mean reflected stochastic differential equations with two constraints}
\author{Adrian Falkowski and Leszek S\l omi\'nski
\footnote{Corresponding author. E-mail address: leszeks@mat.umk.pl;
Tel.: +48-566112954; fax: +48-566112987.}\\
 \small Faculty of Mathematics and Computer Science,
Nicolaus Copernicus University,\\
\small ul. Chopina 12/18, 87-100 Toru\'n, Poland}
\date{}
\maketitle
\begin{abstract}
We study  the problem of the existence, uniqueness and stability
of solutions of reflected stochastic differential equations (SDEs)
with a minimality condition  depending on the law of the solution
(and not on the paths). We require that some functionals depending
on the law of the solution  lie between  two given  c\`adl\`ag
constraints. Applications to
investment models with constraints are given.
\end{abstract}

\section{Introduction}

In this paper we consider  reflected SDEs of the form
\begin{equation}\label{eq1.1}
X_t = X_0 +\int_0^t f(s,X_{s-})\,dM_s+\int_0^tg(s,X_{s-})\, dV_s
+ k_t,\quad t\in\Rp,
\end{equation}
In (\ref{eq1.1}),   $f,g:\Rp\times\R\rightarrow\R$ are continuous
functions, $M$ is a martingale, $V$ is an adapted process of
bounded variation and  $k$ is a deterministic function which for
given two sided Lipschitz continuous function
$h:\Rp\times\R\rightarrow\R$ compensates reflections of the
functional of $X$ of the form $Eh(t,X_t)$, $t\in\Rp$,  on
c\`adl\`ag constraints $l$, $u$ such that $l_t\leq u_t$, $t\in\Rp$
(for a precise definition see Section 3).

Reflected SDEs were introduced by Skorokhod \cite{Sk} in the case
where $l=0$, $u=+\infty$,  $M$ is a standard Wiener process and
$V_t=t$, $t\in\Rp$.   The minimality condition in Skorokhod's
equation was depending on the paths of the solution $X$. This
implies that the compensating reflection part is a nondecreasing
stochastic process  and $X_t\geq 0$, $t\in\Rp$. Since the
pioneering Skorokhod's work reflected SDEs have been intensively
studied by many authors and his results were generalized to larger
classes of constraints  and larger classes of driving processes
(see, e.g.,
\cite{BKR,ChL,DN,KLRS,rutkowski:1980,Sl1,sl-wo/10,sl-wo/13,ta}).
In the all mentioned above papers the minimality condition, which
characterizes  the compensating reflections part,   depends  on
the paths of the solution $X$.

In recent papers by Briand, Elie and Hu \cite{briand:2018} and
Briand Ghannoum and Labart \cite{briand:2018a}  new type of
reflected SDEs was introduced. They were motivated by the mean
field game theory. In this new type of equations the compensating
reflection part depends not on the paths but on the law of the
solution $X$. In this case the compensating reflection part is a
nondecreasing function  and $EX_t\geq0$, $t\in\Rp$, or more
generally, $Eh(t,X_t)\geq0$, $t\in\Rp$,  for given two sided
Lipschitz continuous function $h:\Rp\times\R\rightarrow\R$. This
in fact means that the  mean minimality condition is considered.
A similar problem was  considered by Djehiche, Elie and Hamad\`ene
\cite{ha}. In the present paper, we generalize this type of
reflected SDEs to the case of equations of the form  (\ref{eq1.1})
with two c\`adl\`ag constraints  $l$, $u$ such that
\begin{equation}
\label{eq1.2}Eh(t,X_t)\in[l_t,u_t],\quad t\in\Rp,
\end{equation}
and the compensating reflection part $k$ is not necessarily of
bounded variation.

The paper is organized as follows.  In  Section 2, for given
c\`adl\`ag integrable  process $Y$ we consider the Skorokhod
problem for $Y$ with mean minimality condition. Its solution is a
pair $(X,k)$ such that
\begin{equation}
\label{eq1.3} X_t=Y_t+k_t,\quad t\in\Rp,
\end{equation}
and moreover, (\ref{eq1.2}) is satisfied and $k$ is a
deterministic  function not necessarily of bounded variation. We
observe that  $(X,k)$ is strictly connected with the solution of
appropriately defined classical deterministic Skorokhod problem.
Using this  we prove the existence and uniqueness of solutions of
(\ref{eq1.3}) and provide  an explicit   formula for the function
$k$. We also show  Lipschitz continuity of the mapping
$(Y,h,l,u)\mapsto(X,k)$. More precisely, we prove that if
$(X^i,k^i)$ is a solution associated with an integrable process
$Y^i$, function $h^i$ and barriers $l^i,u^i$, $i=1,2$, then there
exists $C>0$ such that for every $q\in\Rp$,
\[\sup_{t\leq q}|k^1_t-k^2_t|\leq
C\big(\parallel h^1-h^2\parallel_{[0,q]\times\R}
+ \sup_{t\leq q}E|Y^1_t-Y^2_t|+\sup_{t\leq q}
\max(|l^1_t-l^2_t|,|u^1_t-u^2_t|)\big)\]
and
\[E\sup_{t\leq q}|X^1_t-X^2_t|
\leq (C+1)E\sup_{t\le q}|Y^1_t-Y^2_t|+ C\big(\parallel
h^1-h^2\parallel_{[0,q]\times\R}+\sup_{t\leq
q}\max(|l^1_t-l^2_t|,|u^1_t-u^2_t|)\big),\] where $\parallel
h^1-h^2\parallel_{[0,q]\times\R}
=\sup_{(t,x)\in[0,q]\times\R}|h^1(t,x)-h^2(t,x)|$. We also give
results on the  Skorokhod problem with mean minimality condition
and one  barrier: lower (if $u=+\infty$) and upper (if
$l=-\infty$).

In Section 3, we study the general problem of stability of
solutions  to (\ref{eq1.3}) with respect to the convergence of
associated processes  and barriers. We give conditions ensuring
stability with respect to the convergence in law  and in
probability in the Skorokhod topology $J_{1}$. As an application,
we propose a practical scheme  of approximations of (\ref{eq1.3})
based on discrete approximations of the barriers $l,u$ and the
process $Y$.

Section 4 is devoted to the study  of weak and strong solutions of
(\ref{eq1.1}). We prove the existence and uniqueness of a strong
solution  of (\ref{eq1.1}) provided that $f,g$ satisfy the linear
growth condition and are  Lipschitz continuous and   show by
example that equations of the form (\ref{eq1.1}) can be useful in
the study of investment models with constraints. Moreover,  we
prove that the solution  of (\ref{eq1.1}) can be  approximated by
a discrete scheme constructed with the analogy to the Euler scheme
and  we prove its convergence in probability in the Skorokhod
topology $J_{1}$. We also show  some stability results for
solutions of the form
\begin{equation}\label{eq4.9}
X^n_t =  X^n_0 +\int_0^t f^n(s,X^n_{s-})\,d M^n_s
+\int_0^tg^n(s,X^n_{s-})\, dV^n_s + k^n_t,\quad t\in\Rp,
\end{equation}
under the assumption that the sequence of driving martingales
$\{M^n\}$ and processes with locally bounded variations $\{V^n\}$
satisfy the condition corresponding to the so-called condition
(UT) introduced by Stricker \cite{st} (see also \cite{jmp}). As a
consequence,  we prove existence of weak solution of the SDE
(\ref{eq1.1}) provided that  $f,g$ are continuous and satisfy the
linear growth condition. In  case (\ref{eq1.1}) has the weak
uniqueness property, we  formulate a theorem on  convergence of
solutions in  law  in the Skorokhod topology $J_{1}$.

In the paper, we will use the following notation.  $\Dd$ is the
space of c\`adl\`ag  mappings $x:\Rp\to\R^d$, i.e. mappings which
are right continuous and admit left-hands limits
 equipped
with the Skorokhod topology  $J_1$ (for the definition and many
useful results on  $J_1$ topology see, e.g., \cite{js}).  Every
process  $X$  appearing in the sequel is assumed to have
trajectories in the space $\D$. We denote by $\Ll^1$ the space of integrable random variables and by $\Ddet$ the space of
$({\cal F}_t)$-adapted process $X$ such that for every $q\in\Rp$
the family of random variables $\{X_t,t\leq q\}$ is uniformly
integrable. Note that our class $\Ddet$  is larger than usually
considered  Doobs class. For a semimartingale $X$,  $[X]$ stands
for the quadratic variation process of $X$ and $\langle X\rangle_t
$ stands for  the predictable compensator
 of $[X]$. For a process with locally finite
variation $K$ we denote by $|K|_t$ its total variation $[0,t]$. If
additionally $|K|$ is locally integrable, then $\widetilde{K}$
stands for the  predictable compensator of $K$.

For $x\in\Dd$, $t>0$, we set $x_{t-}=\lim_{s\uparrow t}x_s$,
$\Delta x_t=x_t-x_{t-}$ and $v_{p}(x)_{[a,b]} = \sup_\pi
\sum_{i=1}^m |x_{t_i}-x_{t_{i-1}}|^p <\infty$, where the supremum
is taken over all subdivisions $\pi=\{a=t_0<\ldots<t_n=b\}$ of
$[a,b]$. $V_p(x)_{[a,b]}=(v_p(x)_{[a,b]})^{1/p}$ and $\bar
V_p(x)_{[a,b]}=V_p(x)_{[a,b]}+|x_a|$ is the usual variation norm.
For simplicity of notation, we write $v_p(x)_q=v_p(x)_{[0,q]}$,
$V_p(x)_q= V_p(x)_{[0,q]}$ and $\bar V_p(x)_q= \bar
V_p(x)_{[0,q]}$. Note that $V_1(x)_q=v_1(x)_q=|x|_q$, $q\in\Rp$.
Recall also that for  $\eta>0$ and $q\in\Rp$ the number $N_{\eta}$
of $\eta$-oscillations of $x$  on $[0,q]$ is the largest  integer
$k$ such that one can find  $0\leq t_1\leq t_2\leq\dots\leq
t_{2k-1}\leq t_{2k}\leq q$ satisfying
$|x_{t_{2i-1}}-x_{t_{2i}}|>\eta$, $i=1,2,\dots,k$.

By  ${\displaystyle \arrowd}$ and ${\displaystyle\arrowp}$ we
denote the convergence in law and in probability, respectively.

\section{The Skorokhod problem with mean minimality condition}

Let $ l,u\in\D$ be such that $l\le u$, $(\Omega, {\cal F}, ({\cal
F}_t), P)$ be a filtered probability space and
$h:\Rp\times\R\rightarrow \R$ be a
$\B(\Rp)\times\B(\R)$-measurable function for which there exists a
constant $\Mh>0$ such that
\begin{equation}\label{eq2.1}
|h(t,x)|\le \Mh(1+|x|),\quad t\in\Rp,\, x\in\R.\end{equation}
\begin{definition}\label{def1}Let $Y$ be an $({\cal F}_t)$-adapted process with trajectories in $\D$,
$l,u\in\D$, $l\leq u$. Let  $h:\times\Rp\times\R\rightarrow \R$
satisfy (\ref{eq2.1}) and  $Eh(0,Y_0)\in[l_0,u_0]$. We say that a
pair $(X,k)\in{\cal D}\times\D$  with $k_0=0$ is a solution of the
Skorokhod problem) with mean minimality condition and  two
constraints)
 associated with $h,Y,l,u$  ($(X,k)=\SP{Y}$ for short) if
\begin{enumerate}[\bf(i)]
\item $X_t=Y_t+k_t$,
\item $Eh(t,X_t)\in[l_t,u_t],\quad t\in\Rp$,
\item for every $0\leq t\leq
q$,
\begin{eqnarray*}
k_q-k_t\ge 0,&&\quad\mbox{\rm if}\,\,Eh(s,X_s)<u_s\,\,\mbox{\rm for all}\,\,s\in(t,q],\\
k_q-k_t\le 0,&&\quad\mbox{\rm if}\,\,Eh(s,X_s)>l_s\,\,\mbox{\rm
for all}\,\,s\in(t,q],
\end{eqnarray*}
and for every $t\in\Rp$, $
\Delta k_t\ge 0$, if $Eh(t,X_t)<u_t$ and $\Delta k_t\le 0$, if
$Eh(t,X_t)>l_t$.
\end{enumerate}
\end{definition}
We consider the following assumption on  $h$.

\begin{enumerate}
\item[(H)] $h$ satisfies (\ref{eq2.1}),
$x\mapsto h(t,x)$ is strictly increasing for $t\in\Rp$
 and there exist constants  $\Lh,\Ch,\ch>0$ such that
\begin{equation}\label{eq2.2}|h(t,x)-h(s,x)|\le \Lh|t-s|,
\quad t,s\in\Rp,\,x\in\R\end{equation}
and
\begin{equation}\ch|x-y|\le |h(t,x)-h(t,y)|\le \Ch|x-y|,
\quad t\in\Rp,\,x,y\in\R.\label{eq2.3}
\end{equation}
\end{enumerate}

For given  $t\in\Rp$,  $Y\in\Ll^1$ and $h$ satisfying
{(H)} we define new map $H(t,\cdot,Y):\R\to\R$ by
\begin{equation}\label{eq2.4}H(t,z,Y)=Eh(t,Y-EY+z),\quad
z\in\R.\end{equation} By (\ref{eq2.1}), $H(t,\cdot,Y)$ is well
defined. Moreover, by {(H)}, it is strictly increasing, continuous
and $\lim_{z\to-\infty}H(t,z,Y)=-\infty$,
$\lim_{z\to+\infty}H(t,z,Y)=+\infty$. Hence there exists strictly
increasing and continuous inverse map $H^{-1}(t,\cdot,Y):\R\to\R$.
Clearly,  for every $z\in\R$,
\begin{equation}
\label{eq2.5}H^{-1}(t,z,Y)=\bar z\quad{\rm iff}\quad
Eh(t,Y-EY+\bar z)=z.
\end{equation}
It is obvious that if $h(t,x)=x$, $x\in\R$,  hen
$H^{-1}(t,z,Y)=z$, $z\in\R$.
\begin{lemma}
Assume  \mbox{\rm(H)} and  let $Y=(Y_t)\in{\cal D}$. If $\bar
z=(\bar z_t)\in\D$, then \[z=(z_t=H(t,\bar z_t,Y_t))\in\D.
\]
Similarly, if $z=(z_t)\in\D$, then
\[\bar z=(\bar z_t=H^{-1}(t,z_t,Y_t))\in\D.
\]
\label{lem1}
\end{lemma}
\begin{proof}
It is easy to observe that by the Lebesgue dominated convergence
theorem, the function $t\mapsto z_t=Eh(t,Y_t-EY_t+\bar z_t)$ is
c\`adl\`ag. To check the second conclusion assume that
$z=(z_t)\in\D$. If $t_n\to t$, $t_n\geq t$, then from the right
continuity of $z,$ and $Y$, $z_{t_n}\to z_t$ and $Y_{t_n}\to Y_t$
$P$-a.s., which implies that the sequence $\{\bar z_{t_n}\}$ is
bounded (if not, there exists a subsequence $(n')\subset(n)$ such
that $\lim_{t_{n'}}|Eh(t_{n'},Y_{t_{n'}}-EY_{t_{n'}}+\bar
z_{t_{n'}})|=+\infty$). Consequently, we may and will assume that
for some subsequence $(n')\subset(n)$, $\bar z_{t_{n'}}\to z'$.
Then using once again  the Lebesgue dominated convergence shows
that $z'$ satisfies the equation
\[
z_t=Eh(t,Y_t-EY_t+z'),
\]
which implies that $\bar z_t=z'$ and completes the proof of the
right continuity of $\bar z=(\bar z_t)$. Similarly we show that if
$t_n\to t$, $t_n< t$, then there exists a limit $\bar z_{t-}$ of
$\{\bar z_{t_n}\}$ satisfying
\[z_{t-}=Eh(t,Y_{t-}-EY_{t-}+\bar z_{t-}).
\]
\end{proof}
\begin{theorem}
Assume  \mbox{\rm(H)}.  If $Y=(Y_t)\in{\cal D}$,  $l,u\in\D$ and
$Eh(0,Y_0)\in[l_0,u_0]$, then there exists a unique  solution of
the Skorokhod problem $(X,k)=\SP{Y}$ . Moreover,
\begin{equation}\label{eq2.6}
k_t=-\max(0\wedge\inf_{0\leq u\leq t}(EY_u-\bar l_u),
\sup_{0\leq s\leq t}[(EY_s-\bar u_s)
\wedge\inf_{0\leq u\leq t}(EY_u-\bar l_u)]),\quad t\in\Rp,
\end{equation}
where $\bar l=H^{-1}\circ l$, $\bar u=H^{-1}\circ u$, i.e. $\bar
l_t=H^{-1}(t,l_t,Y_t)$, $\bar u_t=H^{-1}(t,u_t,Y_t)$, $t\in\Rp$.
\label{thm1}
\end{theorem}
\begin{proof}
In the proof we use  results from the theory of  deterministic
Skorokhod problem (see Appendix). First we show the existence of
solutions of the  Skorokhod problem with mean minimality
condition. Set $y_t=EY_t$, $t\in\Rp$, and observe that
$y=(y_t)\in\D$. By Lemma \ref{lem1},  $\bar l=H^{-1}\circ l\in\D$,
$\bar u=H^{-1}\circ u\in\D$. Since $H^{-1}$ is strictly
increasing,
\[EY_0=H^{-1}(0,Eh(0,Y_0),Y_0)\in[\bar l_0,\bar u_0].\]
By Theorem \ref{thm5.1}, there exists a unique solution  of the
Skorokhod problem $(x,k)=SP_{\bar l}^{\bar u}(y)$ such that for
every $0\leq t\leq q$,
\begin{eqnarray*}
k_q-k_t\geq0&&\quad\mbox{\rm if
}x_s<H^{-1}(s,u_s,Y_s)\,\,\mbox{\rm for
all }s\in(t,q],\\
k_q-k_t\leq0&&\quad\mbox{\rm if
}x_s>H^{-1}(s,l_s,Y_s)\,\,\mbox{\rm for all
}s\in(t,q],
\end{eqnarray*}
and for every $t\in\Rp$, $ \Delta
k_t\geq0$ if $x_t<H^{-1}(t,u_t,Y_t)$ and $\Delta k_t\leq0$ if
$x_t>H^{-1}(t,l_t,Y_t)$.

Set
\[
X_t=Y_t+k_t,\quad t\in\Rp,
\]
and note that $(X,k)=\SP{Y}$.
Indeed, since $H$ is strictly increasing,
\[
Eh(t,X_t)=Eh(t,Y_t+k_t)=H(t,x_t,Y_t) \in[H\circ\bar l_t,H\circ
u_t]=[l_t,u_t],\quad t\in\Rp,
\]
which  gives  condition (ii) from Definition \ref{def1}. To check
condition (iii) it is sufficient to observe that for every
$s\in\Rp$,
\[x_s<H^{-1}(s,u_s,Y_s)\quad {\rm iff}\quad Eh(s,X_s)<u_s\]
and
\[x_s>H^{-1}(s,l_s,Y_s)\quad {\rm iff}\quad Eh(s,X_s)>l_s.\]

It is easy to prove that every solution  of the Skorokhod  problem
with mean minimality condition is of the form given above. Indeed,
using the arguments used previously one can check that if
$(X',k')=\SP{Y}$, then $(x',k')=SP_{\bar l}^{\bar u}(y)$, where
$x'_t=H^{-1}(t,Eh(t,X'_t),Y_t)$, $t\in\Rp$. By the uniqueness of
the deterministic  Skorokhod problem (see Theorem \ref{thm5.1}),
$(x',k')=(x,k)$, which implies  that $X'_t=Y_t+k_t=X_t$,
$t\in\Rp$. Finally, observe that the form of the compensating
reflection part $k$ follows easily from (\ref{eq5.1}).
\end{proof}

\begin{proposition}
\label{prop1} Assume that $h^i$, $i=2$, satisfy \mbox{ \rm(H)} and
$Y^i\in{\cal D}$, $i=1,2$,  are processes defined on the same
probability space If   $l^i,u^i\in\D$  are such that $l^i\leq
u^i$, $l^i_0\leq Eh^i(0,Y_0)\leq u^i_0$,  $(X^i,
k^i)=\mathbb{SP}^{u^i}_{l^i}(h^i,Y^i)$, $i=1,2$, and
$C_h=\max((c_h^1)^{-1},2c_h^2(c^1_h)^{-1}+1)$, then  for every
$q\in\Rp$,
\begin{equation}
\label{eqprop}\sup_{t\leq q}|k^1_t-k^2_t|\leq
C_h\big(\parallel h^1-h^2\parallel_{[0,q]\times\R}
+ \sup_{t\leq q}E|Y^1_t-Y^2_t|+\sup_{t\leq q}
\max(|l^1_t-l^2_t|,|u^1_t-u^2_t|)\big)
\end{equation}
and
\begin{eqnarray}\label{eqprop1}E\sup_{t\leq q}|X^1_t-X^2_t|
&\leq& (C_h+1)E\sup_{t\le q}|Y^1_t-Y^2_t|+
C_h\big(\parallel h^1-h^2\parallel_{[0,q]\times\R}\nonumber\\
&&\qquad+\sup_{t\leq
q}\max(|l^1_t-l^2_t|,|u^1_t-u^2_t|)\big),
\end{eqnarray} where
$\parallel h^1-h^2\parallel_{[0,q]\times\R}
=\sup_{(t,x)\in[0,q]\times\R}|h^1(t,x)-h^2(t,x)|$.
\end{proposition}
\begin{proof} From the proof of Theorem \ref{thm1}  we know
that $k^i$ are compensation reflection parts of  solutions of the
deterministic  Skorokhod problems $SP_{\bar l^i}^{\bar u^i}(y^i)$
with $y^i_t=EY^i_t$,  $\bar l^i_t=H_i^{-1}(t,l^i_t,Y^i_t)$, $\bar
u^i_t=H_i^{-1}(t,u^i_t,Y^i_t)$, $t\in\Rp$, $i=1,2$. By
(\ref{eq5.3}),
\begin{equation}\label{eq2.7}
 \sup_{t\leq q}|k^1_t-k^2_t|\leq\sup_{t\leq q}|EY^1_t-EY^2_t|
+\sup_{t\leq q}\max(|\bar l^1_t-\bar l^2_t|,|\bar u^1_t-\bar u^2_t|).
\end{equation}
Fix $t\in[0,q]$. To estimate $|\bar l^1_t-\bar l^2_t|$ observe that
\begin{eqnarray*}
l^1_t-l^2_t&=&Eh^1(t,Y^1_t-EY^1_t+\bar l^1_t)-Eh^2(t,Y^2_t-EY^2_t+\bar l^2_t)\\
&=&Eh^1(t,Y^1_t-EY^1_t+\bar l^1_t)-Eh^2(t,Y^1_t-EY^1_t+\bar l^1_t)\\
&&\,+Eh^2(t,Y^1_t-EY^1_t+\bar l^1_t)-Eh^2(t,Y^2_t-EY^2_t+\bar l^1_t)\\
&&\,+Eh^2(t,Y^2_t-EY^2_t+\bar l^1_t)-Eh^2(t,Y^2_t-EY^2_t+\bar l^2_t)\\
&=&I^1_t+I^2_t+I^3_t.
\end{eqnarray*}
Without loss of generality we may assume that $\bar l^1_t>\bar
l^2_t$. Then, by (\ref{eq2.3}), there exists $C\in[c_h^1,c_h^2]$
such that $I^3_t=C(\bar l^1_t-\bar l^2_t)$ , which implies that
\begin{eqnarray*}|\bar l^1_t-\bar l^2_t|&\leq& (1/c_h^1)(|I^1_t|+|I^2_t| +|l^1_t-l^2_t|)\\
&\leq&(1/c_h^1)(\sup_{x\in \R}|h^1(t,x)-h^2(t,x)|+
2c_h^2E|Y^1_t-Y^2_t| +|l^1_t-l^2_t|).\end{eqnarray*} Similarly we
estimate $|\bar u^1_t-\bar u^2_t|$.  From the above and
(\ref{eq2.7}) we deduce (\ref{eqprop}). Estimete (\ref{eqprop1})
easily follows from (\ref{eqprop}).
\end{proof}

\begin{corollary}
\label{cor1}Assume \mbox{\rm ({H})}.  If $Y\in{\cal D}$,
$l,u\in\D$  are such that $l\leq u$, $l_0\leq Eh(0,Y_0)\leq u_0$,
$(X, k)=\mathbb{SP}^u_l(h,Y)$, then for  every $t,q\in\Rp$ such
that $t\leq q$,
\begin{equation}
\label{eqcor}\sup_{t\leq s\leq q}|k_s-k_t|\leq
C_h\big( \sup_{t\leq s\leq q}E|Y_s-Y_t|+\lambda_h(q-t)+\sup_{t\leq s\leq q}\max(|l_s-l_t|,|u_s-u_t|)\big)\end{equation}
and\begin{eqnarray}\label{eqcor1}E\sup_{t\leq s\leq q}|X_s-X_t|&\leq&\nonumber
(C_h+1)E\sup_{t\leq s\leq q}|Y_s-Y_t|\\
&&\qquad+C_h\big(\lambda_h(q-t)+ \sup_{t\leq s\leq q}\max(|l_s-l_t|,|u_s-u_t|)\big).\end{eqnarray}\end{corollary}
\begin{proof}
Fix $t\leq q$ and in Proposition  \ref{prop1} put $Y^1=Y$,
$l^1=l$, $u^1=u$, $h^1=h$ and $Y^2=Y_{\cdot\wedge t}$,
$l^2=l_{\cdot\wedge t}$, $u^2=u_{\cdot\wedge t}$,
$h^2=h_{\cdot\wedge t}$. To complete the proof it is sufficient to
observe that by (\ref{eq2.2}),
\[
\sup_{(t,x)\in[0,q]\times\R}|h^1(t,x)-h^2(t,x)|
\leq \sup_{(s,x)\in[t,q]\times\R}|h(s,x)-h(t,x)|\leq \lambda_h(q-t).
\]
\end{proof}

\begin{remark}\label{rem_dwie}
(a) In the above definition we can replace condition (iii) with
the following one:
\begin{enumerate}
\item [{\bf (iii')}]for every $t\le q\in\Rp$
such that $\inf_{s\in[t,q]}(u_s-l_s)>0$, $k$ is a function of
bounded variation on $[t,q]$ and
\begin{equation}\label{eq_wrminidwie}
\int_{[t,q]} [Eh(s,X_s)-l_s]\,dk_s\le 0  \text{ and }\int_{[t,q]}
[Eh(s,X_s)-u_s]\,dk_s\le0.
\end{equation}
\end{enumerate}

Simple calculation shows that (i), (ii), (iii) and (i), (ii),
(iii') are equivalent.

(b) Note that by \eqref{eq_wrminidwie}, if $\Delta k_t>0$, then
$Eh(t,X_t)=l_t$ and hence $X_t=Y_t-EY_t+\bar l_t$. Similarly, if
$\Delta k_t<0$, then $Eh(t,X_t)=u_t$ and $X_t=Y_t-EY_t+\bar u_t$.
Consequently,
\begin{equation}\label{eq_kwzor}
k_t=\max(\min(k_{t-},\bar u_t-EY_t),\bar l_t-EY_t)),\quad t\in\Rp.
\end{equation}
\end{remark}

Under the additional assumption that  the compensating reflection
part $k$ has  bounded variation the minimality condition {\bf
(iii)} can be written in the following simpler form: for every
 $t\in\Rp$,
\[
\int_{0}^t [Eh(s,X_s)-l_s]\,dk^+_s=0  \text{ and }\int_{0}^t
[u_s-Eh(s,X_s)]\,dk^-_s=0,\]
 where $k^{(+)}$, $k^{(-)}$ are nondecreasing right continuous
functions with $k_0=k^{(+)}_0=k^{(-)}_0=0$ such that $k^{(+)}$
increases only on $\{t;Eh(t,X_t)={l}_t\}$ and  $k^{(-)}$ increases
only on $\{t;Eh(t,X_t)={u}_t\}$. If the barriers $l,u\in\D$
satisfy the  condition
\begin{equation}\label{eq2.8}
\inf_{t\leq q}({u}_t-{l}_t)>0,\quad q\in\Rp,
\end{equation}
then it is possible to estimate the variation of  $k$ using
Proposition \ref{prop5}.

\begin{corollary}\label{cor2}
If $(X, k)=\mathbb{SP}^{u}_{l}(h,Y)$, then for any $q\in\Rp$ and
$\eta$ such that $0<2\eta\leq\inf_{t\leq q}({u}_t-{l}_t)/3$ we
have
\begin{equation}\label{eq2.9}|k|_q\leq
6(N_{\eta}(y,q)+N_{\eta}(\bar l,q)+N_{\eta}(\bar u,q)+1)(\sup_{
t\leq q}|y_t|+\sup_{ t\leq q}\max(|\bar l_t|,|\bar
u_t|)),\end{equation} where $y_t=EY_t$,  $\bar
l_t=H^{-1}(t,l_t,Y_t)$ and $\bar u_t=H^{-1}(t,u_t,Y_t)$,
$t\in\Rp$.
\end{corollary}

Similarly to the case of the deterministic Skorokhod problem with
$u=+\infty$ or $l=-\infty$ the definitions of solutions and  forms
of $k$ are much simpler. We start with the case of one lower
barrier.
\begin{definition}
\label{def2}
Let $Y$ be an $({\cal F}_t)$-adapted
process with trajectories in $\D$,
 $l\in\D$ and let $h:\Rp\times\R\rightarrow \R$
satisfy (\ref{eq2.1}) and  $Eh(0,Y_0)\geq l_0$. We say that a pair
$(X,k)\in{\cal D}\times\D$  with $k_0=0$ is a solution to the
Skorokhod problem (with mean minimality condition and a lower
barrier) associated with $h,Y,l$  ($(X,k)=\SPl{Y}$ for short) if
\begin{enumerate}[\bf(i)]
\item $X_t=Y_t+k_t$,
\item $Eh(t,X_t)\geq l_t,\quad t\in\Rp$,
\item $k$ is nondecreasing and
\[\int_{0}^t [Eh(s,X_s)-l_s]\,dk_s=0,\quad t\in\Rp.\]
\end{enumerate}
\end{definition}

Using the arguments from the proof of Theorem \ref{thm1} and
(\ref{eq5.7}) we obtain the following corollary.
\begin{corollary}
Assume \mbox{\rm{(H)}}. If $Y=(Y_t)\in{\cal D}$, $l\in\D$ and
$Eh(0,Y_0)\geq l_0$, then there exists a unique solution of the
Skorokhod problem $(X,k)=\SPl{Y}$. Moreover,
\begin{equation}\label{eq2.10}
k_t=\sup_{s\leq t}(\bar l_s-EY_s)^+,\quad t\in\Rp,\end{equation}
where  $\bar l_t=H^{-1}(t,l_t,Y_t)$, $t\in\Rp$. \label{cor3}
\end{corollary}
Note that
\[(\bar l_s-EY_s)^+=\inf\{x\geq0;E h(s,Y_s+x)\geq l_s\},\quad
s\in\Rp,\] which means that our formula for $k$ coincides with the
formula considered in \cite{briand:2018,briand:2018a}.

\begin{definition}
\label{def3}
Let $Y$ be an $({\cal F}_t)$-adapted process with trajectories in $\D$,
 $u\in\D$,  and let$h:\Rp\times\R\rightarrow \R$
satisfy (\ref{eq2.1}) and  $Eh(0,Y_0)\leq u_0$. We say that a pair
$(X,k)\in{\cal D}\times\D$  with $k_0=0$ is a solution to the
Skorokhod problem (with mean minimality condition and a upper
barrier) associated with $h,Y,u$  ($(X,k)=\SPu{Y}$ for short) if
\begin{enumerate}[\bf(i)]
\item $X_t=Y_t+k_t$,
\item $Eh(t,X_t)\leq u_t,\quad t\in\Rp$,
\item $k$ is nonincreasing and
\[\int_{0}^t [u_s-Eh(s,X_s)]\,dk_s=0,\quad t\in\Rp.
\]
\end{enumerate}
\end{definition}
By the arguments from the proof of Theorem \ref{thm1} and
(\ref{eq5.8}) we obtain the the following counterpart to Corollary
\ref{cor3}.

\begin{corollary}
Assume \mbox{\rm{(H)}}.  If $Y=(Y_t)\in{\cal D}$, $u\in\D$ and
$Eh(0,Y_0)\leq u_0$, then there exists a unique solution of the
Skorokhod problem $(X,k)=\SPu{Y}$. Moreover,
\begin{equation}\label{eq2.11}
k_t=-\sup_{s\leq t}(\bar u_s-EY_s)^-,\quad t\in\Rp,
\end{equation}
where  $\bar u_t=H^{-1}(t,u_t,Y_t)$, $t\in\Rp$. \label{cor4}
\end{corollary}
One can note that
\[(\bar u_s-EY_s)^-=\inf\{x\geq0;E h(s,Y_s-x)\leq u_s\},\quad
s\in\Rp.
\]

\section{Stability of solutions}

In this section,  we consider a sequence of processes $\{Y^n\}$
such that for every $q\in\Rp$,
\begin{equation}
\label{eq3.1} \{Y^n_t;\, t\leq q,\,n\in\N\}\,\,\mbox{\rm is
uniformly integrable.}
\end{equation}
In the sequel, we will use the  notion of the convergence in $\Dd$
for different $d\in\N$. We recall that
$(x^{n,1},...,x^{n,d})\lra(x^1,...,x^d)$ in  $\Dd$  if
for  every $t\in\Rp$  there exists a  sequence $t_n\to t$ such
that  for all $t'_n\to t$ and $t''_n\to t$ satisfying
$t'_n<t_n\leq t''_n$, $ n\in\N$, we have
\begin{equation}\label{eq3.2} x^{n,i}_{t'_n}
\lra x^i_{t-}\quad\mbox{\rm and} \quad x^{n,i}_{t''_n}\lra
x^i_t,\,\,i=1,\dots,d.
\end{equation}
Note that (\ref{eq3.2}) implies in particular  that  in the case
of the jump in $t\in\Rp$ in the limit  there exists a common
sequence $t_n\to t$  such that $\Delta x^{n,i}_{t_n}\lra\Delta
x^i_t$, $i=1,\dots,d$.

\begin{lemma}
Assume  \mbox{\rm{(H)}} and  \mbox{\rm(\ref{eq3.1})}. Let $\{\bar
z^n\}\subset\D$ and $z^n=(z^n_t=H(t,\bar z^n_t,Y^n_t))$, $n\in\N$.
If $(Y^n,y^n,\bar z^n)\arrowd (Y,y,\bar z)$  in $\DDD$,  then
\[ (Y^n,y^n,\bar z^n,z^n)\arrowd (Y,y,\bar z,z)   \quad in\,\,\DDDD.\]
Similarly, let  $\{z^n\}\subset\D$ and $\bar z^n=(\bar
z^n_t=H^{-1}(t,z^n_t,Y^n_t))$, $n\in\N$. If $(Y^n,y^n, z^n)\arrowd
(Y,y, z)$  in $\DDD$,  then
\[ (Y^n,y^n, z^n,\bar z^n)\arrowd (Y,y,z,\bar z)   \quad in\,\,\DDDD.\] \label{lem2}
\end{lemma}
\begin{proof}
Assume  (\ref{eq3.2})  for the sequence $\{(Y^n,y^n,\bar z^n)\}$,
i.e. assume that for  every $t\in\Rp$  there exists a  sequence
$t_n\to t$ such that for all $t'_n\to t$, $t''_n\to t$ satisfying
$t'_n<t_n\leq t''_n$, $ n\in\N$, we have
\[
(Y^{n}_{t'_n},y^n_{t'_n},\bar z^n_{t'_n})\arrowd (Y_{t-},y_{t-},\bar z_{t-})\quad\mbox{\rm and}\quad
(Y^{n}_{t''_n},y^n_{t''_n},\bar z^n_{t''_n})\arrowd
(Y_{t},y_{t},\bar z_{t}).
\]
We will check  that  the same
condition holds true for the sequence $\{(Y^n,y^n,\bar
z^n,z^n)\}$. Fix $t\in\Rp$  and assume that  there exists a
sequence $\{r_n\}$ such that  $r_n\to t$ and
\[(Y^n_{r_n},y^n_{r_n},\bar z^n_{r_n})\arrowd (Y',y',\bar
z')\quad\mbox{\rm in}\,\, \R^3. \] By the  Lebesgue dominated
convergence theorem,
\[
z^n_{r_n}=Eh(r_n,Y^n_{r_n}-EY^n_{r_n}+\bar z^n_{r_n})
\lra Eh(t,Y'-EY'+\bar z').
\]
Clearly  the limit is  equal to $z_t$ if  $r_n\geq t_n$ for all
sufficiently large $n$  (resp. $z_{t-}$ if $r_n<t_n$ for all
sufficiently large  $n$).

To prove the second part of the lemma we first show  that for
every $q\in\Rp$ the sequence $\{\sup_{t\leq q}|\bar z^n_t)|\}$ is
bounded. To get contradiction, suppose that there is a sequence
$r_n\leq q$ such that $\bar z^n_{r_n}\nearrow+\infty $ (resp.
$\searrow-\infty$). Then there exists a subsequence
$(n')\subset(n)$ such  that $r_{n'}\to t$ and
$Y^{n'}_{r_{n'}}\arrowd Y$ for some $t\le q$ and random variable
$Y$. Consequently,
\[ z^{n'}_{r_{n'}}
=Eh(r_{n'},Y^{n'}_{r_{n'}}-EY^{n'}_{r_{n'}} +\bar
z^{n'}_{r_{n'}})\lra+\infty \mbox{ (resp.}\,\,-\infty),
\]
which contradicts the fact that  $\{z^{n'}_{r_{n'}}\}$ has two
possible limit  points $z_t$  and $z_{t-}$. Now we  assume
(\ref{eq3.2})  for the sequence $\{(Y^n,y^n, z^n)\}$. We are going
to check  that  the same condition holds true for $\{(Y^n,y^n,
z^n,\bar z^n)\}$.  Let $\{r_n\}$ be such that $r_n\to t\leq q$.
First we assume that   $r_n\geq t_n$ for all sufficiently large
$n$. In this case $Y^{n}_{r_{n}}\arrowd Y_t$ and $z^n_{r_n}\to
z_t$. Moreover,  there is a finite $z'$ such that for some
subsequence $(n')\subset(n)$, $\bar z_{r_{n'}}\to z'$. Then using
the Lebesgue dominated convergence theorem  shows that $z'$
satisfies the equation
\[z_t=Eh(t,Y_t-EY_t+z').\]
Since the solution of the above equation is unique, by the same
arguments as above we check that in fact $z'$ is the limit of the
sequence $\{\bar z^n_{r_n}\}$. Hence $\bar z^n_{r_n}\to \bar z_t$.
Similarly we show that  if $r_n<t_n$ for all sufficiently large
$n$, then $\bar z_{r_n}\lra \bar z_{t-}$ being a solution to the
equation $z_{t-}=Eh(t,Y_{t-}-EY_{t-}+\bar z_{t-})$.
\end{proof}

\begin{remark}
{\rm Since the sequences $\{z^n\}$, $\{\bar z^n\}$ are
deterministic, it is clear that in the statement of Lemma
\ref{lem2}  one can replace the convergence in law  by the
convergence in probability in $J_1$ or by the convergence $\Ppw$
in $J_1$. }\label{rem4}
\end{remark}

\begin{theorem}
Assume  \mbox{\rm({H})}.  Let $\{Y^n\}$ be a sequence of processes
satisfying \mbox{\rm(\ref{eq3.1})} and $\{l^n\},\{u^n\}$ be
sequences of c\`adll\`ag functions such that $l^n\leq u^n$,
$l^n_0\leq Eh(0,Y_0)\leq u^n_0$, $n\in\N$. Let  $(X^n,
k^n)=\mathbb{SP}^{u^n}_{l^n}(h,Y^n)$, $n\in\N$.
\begin{enumerate}[\bf(i)]
\item If $(Y^n,y^n,l^n,u^n)\arrowd (Y,y,l,u)$ in $\DDDD$, then
\[(X^n,k^n,Y^n,y^n,l^n,u^n)\arrowd (X,k,Y,y,l,u)
\quad\mbox{ in}\,\,\Ddddddd,\]
\item if $(Y^n,y^n,l^n,u^n)\arrowp (Y,y,l,u)$ in $\DDDD$, then
\[(X^n,k^n,Y^n,y^n,l^n,u^n)\arrowp (X,k,Y,y,l,u)
\quad\mbox{ in}\,\,\Ddddddd,
\]
\item if $(Y^n,y^n,l^n,u^n)\lra (Y,y,l,u)$ $\Ppw$ in $\DDDD$, then
\[(X^n,k^n,Y^n,y^n,l^n,u^n)\lra(X,k,Y,y,l,u)\quad \Ppw \mbox{ in}\,\,\Ddddddd,\]
\end{enumerate}
where $(X, k)=\mathbb{SP}^{u}_{l}(h,Y)$.\label{thm2}
\end{theorem}
\begin{proof}By Lemma \ref{lem2}  and (\ref{eq3.2}) it is clear that
\[(Y^n,y^n,l^n,\bar l^n,u^n,\bar u^n)\arrowd (Y,y,l,\bar l,u, \bar u)\quad\mbox{\rm in}\,\,\Ddddddd.\]
By this  and (\ref{eq5.4}),
\[
(x^n,k^n,Y^n,y^n,l^n,\bar l^n,u^n,\bar u^n) \arrowd
(x,k,Y,y,l,\bar l,u, \bar u)\quad\mbox{\rm in}\,\,\Ddddddddd,\]
where $(x^n,k^n)=SP_{\bar l^n}^{\bar u^n}(y^n)$, $n\in\N$  and
$(x,k)=SP_{\bar l}^{\bar u}(y)$, Since $X^n=Y^n+k^n$  and $X=Y+k$,
assertion (i) easily follows. To prove (ii) and (iii) it is
sufficient to use Remark \ref{rem4} and previously used  arguments
with the convergence in probability in $J_1$ or  $\Ppw$ in $J_1$
in place of the convergence in law.
\end{proof}

In the above theorem joint convergence of $\{(Y^n,y^n)\}$ in $\DD$
is assumed. This assumption  is much stronger than the convergence
of the initial sequence  $\{Y^n\}$.
 If $Y^n\arrowd Y$ in $\D$,  then it is clear  only that $y^n_t=EY^n_t\to
y_t=EY_t$ provided that $E|\Delta Y_t|=0$. Unfortunately, the
functional convergence $y^n\to y$ in $\D$ need not  hold.

\begin{example}
\label{ex1} Let $V$ be an arbitrary nondeterministic random
variable. For some $c>0$  put $Y^n_t={\bf 1}_{\{t\geq c+V/n\}}$,
$t\in\Rp$, $n\in\N$. Then of course $Y^n\lra Y$ $\Ppw$ in $\D$,
where $Y_t={\bf 1}_{\{t\geq c\}}$, $t\in\Rp$. On the other hand
$y^n$ do not tend to $y$ in $\D$ because  there in no sequence
$t_n\to c$ such that $\Delta y^n_{t_n}\to\Delta y_c=1$.
\end{example}

However, in important cases one  can deduce joint convergence  of
$\{(Y^n,y^n)\}$ in $\DD$
 from the convergence of $\{Y^n\}$ in $\D$.
\begin{proposition}\label{prop4}
Let $\{Y^n\}$ be a sequence of processes satisfying
\mbox{\rm(\ref{eq3.1})} and such that $Y^n\arrowd Y$ in $\D$. If
$P(\Delta Y_t=0)=1$, $t\in\Rp$, or $\{Y^n\}$ is a sequence of
processes with independent increments, then
\begin{equation}\label{eq3.3} (Y^n,y^n)\arrowd(Y,y)\quad\mbox{ in}\,\,\DD.
\end{equation}
\end{proposition}
\begin{proof} First assume that   $P(\Delta Y_t=0)=1$, $t\in\Rp$.
We will show that in this case  $y^n$ tends to $y$ uniformly  on
compact sets. By (\ref{eq3.1}), for every $t\in\Rp$, $\Delta
y_t=E\Delta Y_t=0$, which  means that the map $t\mapsto y_t$ is
continuous. Therefore, in order to finish the proof of
(\ref{eq3.3}), it is sufficient to observe that  for every $t$ and
every sequence $\{t_n\}$ such that $t_n\to t$,
\[y^n_{t_n}=EY^n_{t_n}\lra EY_t=y_t.\]
The above property is  a consequence of the fact that
$Y^n_{t_n}\arrowd Y_t$, which holds true provided that  $P(\Delta
Y_t=0)=1$.

To prove the second case we will use  Jacod's theorem
\cite[Theorem 1.21]{jacod} which  says that the convergence
$Y^n\arrowd Y$ in $\D$ is equivalent to the joint convergence of
their characteristics $(B^{h,n},\tilde
C^{h,n},f*\nu^n)\lra(B^h,\tilde C^h,f*\nu)$ in $\DDD$, $f\in C_s$.
It is important here that $B^{h,n}_t=EY^{h,n}_t$, $t\in\Rp$,
$n\in\N$ and $h:\R\to\R$ is a continuous function depending on the
parameter  $a>0$ such that $|h|\leq a$  and if $|x|\leq a/2$
(resp. $|x|\geq a$) then $h(x)=x$  (resp. $h(x)=0$)  and
\[Y^{h,n}_t=Y^n_t-\sum_{s\leq t}(\Delta Y^n_s-h(\Delta Y_s))
=Y^n_t-J^{n,a}_t,\quad t\in\Rp,\,n\in\N.\] Since $\{Y^{h,n}\}$ is
a sequence of weakly convergent processes with independent
increments with bounded jumps, the family of random variables
$\{\sup_{t\leq q}|Y^{h,n}_t|\}$ is uniformly integrable. By this
and (\ref{eq3.1}),
$\lim_{a\to\infty}\limsup_{n\to\infty}\sup_{t\leq
q}E|J^{n,a}_t|=0$. Since $y^n_t=B^{h,n}_t+EJ^{n,a}_t$, from
\cite[Theorem 1.21]{jacod} we deduce that $y^n\lra y$ in $\D$.  To
prove (\ref{eq3.3}) it is sufficient to observe that for every
$t\in\Rp$ and every  sequence $\{t_n\}$ such that $t_n\to t$ and
\[(B^{h,n}_{t_n},\tilde C^{h,n}_{t_n},f*\nu^n_{t_n})\lra(B^h_t,\tilde C^h_t,f*\nu_t),\quad (B^{h,n}_{t_n-},\tilde C^{h,n}_{t_n-},f*\nu^n_{t_n-})
\lra(B^h_{t-},\tilde C^h_{t-},f*\nu_{t-})\]
the convergences $Y^n_{t_n}\arrowd Y_t$ and $Y^n_{t_n-}\arrowd Y_{t-}$ also hold true.
\end{proof}

\begin{corollary}\label{cor5}
Let $\{Y^n\}$ be a sequence of processes satisfying
\mbox{\rm(\ref{eq3.1})} and such that $Y^n\arrowd Y$ in $\D$. If
for every $q$ and all sequences $\{s_n\}$, $\{r_n\}$ such that
$0\leq s_n,r_n\leq q$ and $s_n-r_n\to0$ we have
\begin{equation}\label{eq3.4}Y^n_{s_n}-Y^n_{r_n}\arrowp0,
\end{equation}
then \mbox{\rm(\ref{eq3.3})} holds true.
\end{corollary}
\begin{proof}
It is sufficient to observe   (\ref{eq3.4}) implies that function
$y$ is  continuous.
\end{proof}

Condition (\ref{eq3.3}) is also satisfied  in the case where
$Y^n$, $n\in\N$, are  discretizations of $Y$.  This allows us to
approximate
  the solution  $(X,k)=\SP{Y}$ by simple
discretization method. As above set $y=(y_t=EY_t)$ and
$\rho^n_t=k/n$ for $t\in[k/n,(k+1)/n)$. Let $Y^n,y^n,l^n,u^n$ be
discretizations of $Y,y,l,u$, that is $Y^n_t=Y_{\rho^n_t}$,
$y^n_t=y_{\rho^n_t}$, $l^n_t=l_{\rho^n_t}$ and
$u^n_t=u_{\rho^n_t}$ $t\in\Rp$. Set
\begin{equation}\label{eq3.6}
\left\{\begin{array}{ll}
k^n_0&=0,\qquad X^n_0=Y_0,\\[2mm]
k^n_{{(k+1)}/{n}}&= \max\big[\min\big[k^n_{{k}/{n}},
\bar u^{n}_{(k+1)/{n}}-EY_{(k+1)/n}\big],
\bar l^n_{(k+1)/n}-EY_{(k+1)/n})\big],\\[2mm]
X^n_{{(k+1)}/{n}}&= Y_{(k+1)/n}+k^n_{(k+1)/n},
\end{array}
\right.
\end{equation}
where $\bar l_{(k+1)/n}=H^{-1}((k+1)/n,l_{(k+1)/n},Y_{(k+1)/n})$,  $\bar u_{(k+1)/n}=H^{-1}((k+1)/n,u_{(k+1)/n},Y_{(k+1)/n})$ and
 $k^n_t=k^n_{k/n}$, $X^n_t=X^n_{k/n}$ for $t\in[k/n,(k+1)/n)$, $k\in\N\cup\{0\}$, $n\in\N$.

\begin{theorem}\label{thm3}Assume  \mbox{\rm({H})}.
Let $Y\in{\cal D}$ and $l,u\in\D$ be such  that $l\leq u$ and
$l_0\leq Eh(0,Y_0)\leq u_0$. Let $\{Y^n\}$  and $\{l^n\}$,
$\{u^n\}$ be sequences of  discretizations of $Y$ and $l,u$,
respectively,  and $(X^n, k^n)$ be defined by
\mbox{\rm(\ref{eq3.6})}, $n\in\N$. Then  $(X^n,
k^n)=\mathbb{SP}^{u^n}_{l^n}(h,Y^n)$, $n\in\N$  and
\begin{enumerate}\item[{\bf (i)}]
$\displaystyle{(X^n,k^n,Y^n,y^n,l^n,u^n)\lra (X,k,Y,y,l,u)\quad
\Ppw\,\,\mbox{ in}\,\,\Ddddddd},$
\item[{\bf (ii)}] for every $q\in\Rp$,
\[\max_{k;\,k/n\leq q}(|X^n_{k/n}-X_{k/n}|+|k^n_{k/n}-k_{k/n}|)
\lra0\quad \Ppw,\]
\end{enumerate}
where $(X,k)=\SP{Y}$.
\end{theorem}
 \begin{proof}
First note that in the case where $y,l$ or $u$ have a jump in $t$
there is a common sequence $\{t^*_n=\inf\{k/n;k/n\geq t\}\}$ such
that $t^*_n\to t$ and $\Delta Y^n_{t^*_n}\to\Delta Y_t$, $\Delta
y^n_{t^*_n}\to\Delta y_t$, $\Delta l^n_{t^*_n}\to\Delta l_t$ and
$\Delta u^n_{t^*_n}\to\Delta u_t$. Consequently,
\begin{equation}\label{eq3.7}
(Y^n,y^n,l^n,u^n)\lra (Y,y,l,u)\quad \Ppw\,\,\mbox{\rm
in}\,\,\DDDD.\end{equation} Let $(x^n,k^n)$ be the solution of the
deterministic Skorokhod problem associated   with $y^n$, $\bar
l^n$ and $\bar u^n$, $n\in\N$. Then $k^n$ and $X^n=Y^n+k^n$ have
the form given in (\ref{eq3.6}). Moreover, from (\ref{eq3.7}) and
Theorem \ref{thm2}(iii) the assertion (i) easily follows. To prove
(ii) observe  that for every $t\in\Rp$  the convergences $\Delta
\hat X^n_{t^*_n}\to\Delta X_t$, $\Delta \hat k^n_{t^*_n}\to\Delta
k_t$, hold true, where  $\hat X^n$, $\hat k^n$ are discretizations
of $X$, $k$ i.e. $\hat X^n_t=X_{k/n}$, $\hat k^n_t=k_{k/n}$ for
$t\in[k/n,(k+1)/n)$, $k\in\N\cup\{0\}$, $n\in\N$. Therefore
\[(
X^n,\hat X^n,k^n,\hat k^n)\lra (X,X,k,k)\quad \Ppw\,\,\mbox{
in}\,\,\Ddddd,\] which implies that  $X^n-\hat X^n\to0$, $k^n-\hat
k^n\to0$ $\Ppw$  in $\D$. This completes the proof.
\end{proof}

\section{SDEs with mean reflection}

Let $ l,u\in\D$ be such that $l\le u$ and $(\Omega, {\cal F}, ({\cal
F}_t), P)$ be a filtered probability space.
 Let $X_0$ be an ${\cal
F}_0$-measurable integrable random variable,  $V$ be an $({\cal
F}_t)$-adapted process with trajectories in $\D$ such that
$E|V|_q<\infty$ for every $q\in\Rp$, and let $M$ be a square
integrable $({\cal F}_t)$-martingale.

We consider equation with mean reflection of the form
(\ref{eq1.1}).
\begin{definition}\label{def_solutiion}
We  say that a pair $(X,k)$ of $({\cal F}_t)$ adapted processes
with trajectories in $\D$ is a strong solution of (\ref{eq1.1}) if
$(X,k)=\SP{Y}$, where
\[
Y_t=X_0+\int_0^tf(s,X_{s-})\,dM_s+\int_0^tg(s,X_{s-})\,dV_s,\quad
t\in\Rp.
\]
\end{definition}

We will need the following conditions on the coefficients $f,g$.
\begin{enumerate}
\item[(A1)] $f$ and $g$ are  continuous and there exists $\mu>0$ such that
\[|f(t,x)|+|g(t,x)|\leq \mu(1+|x|),\quad t\in\Rp,\,\,x,\,y\in\R.\]
\item[(A2)] There exists $c>0$ such that
\[|f(t,x)-f(t,y)|+|g(t,x)-g(t,y)|\leq c|x-y|,\quad t\in\Rp,\,\,x,\,y\in\R.\]
\end{enumerate}
We will also assume boundedness of  predictable characteristics of
$M,V$ in the following sense.
\begin{enumerate}
\item[(M)] There exists a nondecreasing c\`adl\`ag  function
$m:\Rp\rightarrow\R$ with  $m_0=0$ such that
\[\max(\langle M\rangle_t,\widetilde{|V|}_t)\le m_t,\quad t\in\Rp.\]
\end{enumerate}

Below we give simple examples of processes $M,V$ satisfying  (M).
\begin{example}\label{ex2} (a) Let $Z$ be a semimartingale with
independent increments and bounded jumps (i.e. there exists $a>0$
such that $|\Delta Z|\leq a$). Then $Z$ is  of the form
\begin{equation}\label{eq4.0}Z_t=B_t +M_t,\quad
t\in\Rp,\end{equation} where $B=(B_t=EZ_t)$ is a deterministic
function with locally bounded variation, $|\Delta B|\leq a$ and
$M=(M_t=Z_t-EZ_t)$ is a square integrable martingale, $|\Delta
M|\leq 2a$. In this case $\widetilde{|B|}=|B|$ and $\langle
M\rangle$ are deterministic nondecreasing functions starting from
$0$. Clearly, $M$ and $B$ satisfy (M) with $m_t=\max(|B|_t,\langle
M\rangle_t)$, $t\in\Rp$.

(b)  Let $M,V$ satisfy  (M)  and $H^1,H^2$ be two bounded
predictable processs. Then $\int_0^\cdot H^1_s\,dM_s$,
$\int_0^\cdot H^2_s\,dV_s$ also satisfy (M) with slightly modified
function $m$. \end{example}

We will use the following version of the estimate \cite[Chapter 1,
Section 9 Theorem 5]{LS}: for every square integrable martingale
$M$ and every $t>0$,
\begin{equation}\label{eq4.1}E\sup_{s<t}|M_s|\leq 3E(\langle
M\rangle_{t-})^{1/2}.
\end{equation}

\begin{theorem} \label{thm4}
Assume  \mbox{\rm(H)},  \mbox{\rm(A1)}, \mbox{\rm(A2)} and
\mbox{\rm (M)}. If $l_0\leq Eh(0,X_0)\leq u_0$,  then there exists
a unique strong solution $(X,k)$ of \eqref{eq1.1} such that
$E\sup_{t\leq q} |X_t|<+\infty$, $q\in\Rp$.
\end{theorem}
\begin{proof}
Set $t_1=\inf\{t>0:(C_h+1)c\max(m_t,3(m_t)^{1/2})>1/2\}$, where
$C_h$ is the constant from Proposition \ref{prop1}. In the first
step of the proof we  show the existence and uniqueness of
solutions of  \eqref{eq1.1} on the interval $[0,t_1)$. Set
\[
\Ss^1=\{Y:\, Y\text{ is }(\F_t)\text{-adapted, }
Y_0=X_0,\,\,Y=Y^{t_1-},\,\,E\sup_{t<t_1}|Y_t|<\infty\}
\]
and define the map $\Phi:\, \Ss^1\rightarrow\Ss^1$ by putting
$\Phi(Y)$ to be the first coordinate of the solution of the
Skorokhod problem $\SP{Z}$ with $Z= X_0+\int_0^\cdot
f(s,Y_{s-})\,dM^{t_1-}_s+\int_0^\cdot g(s,Y_{s-})\,dV^{t_1-}_s$
stopped in $t_1-$. First we prove that if $Y\in\Ss^1$, then
$\Phi(Y)\in\Ss^1$. Since $k$ is deterministic it  is sufficient to
prove that $Z\in\Ss^1$. By (\ref{eq4.1}), (M)  and (A1),
\begin{eqnarray*}
E\sup_{t<t_1}|Z_t|&\leq& E|X_0|
+E\sup_{t<t_1}\int_0^t|f(s,Y_{s-})\,dM_s|
+E\int_0^{t_1-}| g(s,Y_{s-})|\,d|V|_s\\
&\leq&E|X_0|+3E(\int_0^{t_1-}|f(s,Y_{s-})|^2\,d\langle M\rangle_s)^{1/2}+E\int_0^{t_1-}| g(s,Y_{s-})|\,d\widetilde{|V|}_s\\
&\leq&E|X_0|+3(m_{t_1-})^{1/2}E\sup_{s< t_1}|f(s,Y_{s-})|+m_{t_1-}E\sup_{s< t_1}|g(s,Y_{s-})|\\
&\leq&E|X_0|+\max(3(m_{t_1-})^{1/2},m_{t_1-})\mu(1+E\sup_{s<t_1}|Y_s|)<+\infty.
\end{eqnarray*}
By Proposition \ref{prop1}, for any $Y,Y'\in\Ss^1$ we have
\begin{align*}
E\sup_t |\Phi(Y)_t-\Phi(Y')_t| & \le (C_h+1)E\sup_{t<t_1} \left|\int_0^{t} f(s,Y_{s-})-f(s,Y'_{s-})\,dM_s\right. \\
&\quad+\left.\int_0^{t_1-}| g(s,Y_{s-})-g(s,Y'_{s-})|\,d|V|_s\right|\\
 &\le (C_h+1)3E\left(\int_0^{t_1-}|f(s,Y_{s-})-f(s,Y'_{s-})|^2\,d\langle M\rangle_s\right)^{1/2}\\
 &\quad+(C_h+1)E\int_0^{t_1-}|g(s,Y_{s-})-g(s,Y'_{s-})|\,d\widetilde{|V|}_s)\\
 &\le (C_h+1)c\max(3(m_{t_1-})^{1/2},m_{t_1-})E\sup_t|Y_t-Y'_t|\\
 &\le \frac 12 E\sup_t|Y_t-Y'_t|.
\end{align*}
By the Banach contraction principle, there exists a unique
solution $(X,k)$ of \eqref{eq1.1} on the interval $[0,t_1)$. By
Remark \ref{rem_dwie} (b) we can get a unique solution on
$[0,t_1]$ by putting
\[k_{t_1}=\max(\min(k_{t_1-},\bar u_{t_1}-EY_{t_1}),\bar l_{t_1}-EY_{t_1})),\quad
X_{t_1}=Y_{t_1}+k_{t_1},\] where
$Y_{t_1}=X_{t_1-}+f(t_1,X_{t_1-})\Delta
V_{t_1}+g(t_1,X_{t_1-})\Delta M_{t_1}$, $\bar
l_{t_1}=H^{-1}(t_1,l_{t_1},Y_{t_1})$  and $\bar
u_{t_1}=H^{-1}(t_1,u_{t_1},Y_{t_1})$. Now we define the sequence
$\{t_k\}$ by setting $t_{k+1}=t_k+\inf\{t>0:(C_h+1)c\min(\hat
m_t,3(\hat m_t)^{1/2})>1/2\}$, where $\hat m_t=m_{t_k+t}-m_{t_k}$,
$k\in\N$. Arguing as above, one can obtain a solution of
\eqref{eq1.1} on $[t_k,t_{k+1}]$. Since $t\mapsto m_t$ is
c\`adl\`ag, $t_k\uparrow\infty$ and the unique solution on $\Rp$
we obtain by  putting together the solutions on intervals
$[t_k,t_{k+1}]$, $k\in\N$.
\end{proof}
The following example shows that solutions of (\ref{eq1.1}) can be applied in investment models with constraints.

\begin{example}
Consider an insurance company with initial capital $x>0$, whose
risk reserve process is given by a semimartingale $J$ with
independent increments and bounded jumps.  The  company is allowed
to invest the risk reserve into the financial market consisting of
the riskless bond $B_t=1$, $t\in\Rp$, and the risky stock whose
price $S$ evolves according to the stochastic differential
equation
\[S_t=S_0+\int_0^tS_ub\,du+\int_0^tS_u\sigma \, dW_u,\quad t\in\Rp,\]
where $W$ is a Wiener process, $\sigma>0$ and $b\in\R$. Let
$V^{k,\pi}$ denote the company wealth portfolio under investment
strategy $(-k,\pi)$, that is
\begin{equation}\label{eq_kapital}
V^{k,\pi}_t=-k_t+\pi_tS_t,\quad t\in\Rp.
\end{equation}
In \cite{ChM} (see also \cite{BR}) the authors suggest that it  is
reasonable to assume that the strategy is deterministic. Following
this suggestion we  restrict ourselves to strategies $(-k,\pi)$
such that $k$ is deterministic and  $(-k,\pi)$ is self-financing,
that is the dynamics of $V^{k,\pi}$ is given by the following
equation
\begin{equation}\label{eq_samo}V^{k,\pi}_t=x+\int_0^t\pi_{u}\,dS_u+\int_0^t\,dJ_u,\quad t\in\Rp.
\end{equation}
Let $X_t=\pi_tS_t$, $t\in\Rp$, be the amount of money  invested in
stock at time $t\in\Rp$. By \eqref{eq_kapital} and
\eqref{eq_samo},  $X$ satisfies the equation
\[
X_t=x+\int_0^t X_{s}b\,ds+\int_0^tX_{s}\sigma\,dW_s +J_t+k_t,\quad
t\in\Rp.
\]

Let $l,u\in\D$, $l\le u$ and $h:\Rp\times \R\rightarrow\R$ be  a
concave function satisfying (H). We consider risk measure imposing
some  restriction on the class of admissible strategies. Namely, a
portfolio is considered admissible if and only if for every
$t\in\Rp$ the amount of money $X_t$ invested in the risky stock
satisfies the following constraints
\begin{equation}\label{eq_admis}
L_t(X_t)\le 0\le U_t(X_t),\quad t\in\Rp,
\end{equation}
where $(L_t)$ and $(U_t)$ are collections of convex risk  measures
given by the formulas
\[
L_t(X_t)=\inf\{k\in \R:Eh(t,X_t+k)\ge l_t\},\quad
U_t(X_t)=\sup\{k\in\R:Eh(t,X_t+k)\le u_t\}.
\]
(recall that a map $\rho:\Ll^1\rightarrow\R$ is a convex risk
measure iff for all $a\in\R$, $\lambda\in[0,1]$ and $X$,
$Y\in\Ll^1$ such that $X\le Y$ we have that $\rho(X+k)=\rho(X)-k$,
$\rho(X)\ge \rho(Y)$ and $\rho(\lambda X+(1-\lambda) Y)\le
\lambda\rho(X)+(1-\lambda)\rho(Y)$).

Note that for every $t\in\Rp$ we have  $L_t(X_t)\le U_t(X_t)$, so
one can say that for every $t\in\Rp$ the risk measure $U_t$ is
more restrictive than $L_t$. The interpretation of
\eqref{eq_admis} is the following.  If position $X$ is very risky
(i.e. $L_t(X_t)>0$), then the company have  to borrow some money
from bank account and buy some amount of stock. On the other hand,
if position $X$ is safe enough (i.e. $U_t(X_t)<0$), then the
company sells some amount of stock and reduces the debt in bank
account. The company is looking for the  minimal strategy in the
sense that if \eqref{eq_admis} is satisfied then there is no money
flow from the bank account to the stock market or in the opposite
direction.

Since for every $\Rp$ the condition $L_t(X_t)\le 0$ (resp.
$U_t(X_t)\ge 0$) is equivalent to the condition $Eh(t,X_t)\ge l_t$
(resp. $Eh(t,X_t)\le u_t$), by Theorem \ref{thm4} there exists a
unique minimal admissible strategy $(-k,\pi)$. More precisely,
there exists a unique solution $(X,k)=\SP{x+\int_0^\cdot
X_{s}b\,ds+\int_0^\cdot X_{s}\sigma\,dW_s+J}$ and if we set $\pi
_t=S^{-1}_tX_t$, $t\in\Rp$,  then the wealth portfolio $V^{k,\pi}$
satisfies \eqref{eq_admis}.
\end{example}

We are able to approximate
  the solution  $(X,k)$ of (\ref{eq1.1}) by a simple
discretization method being  a counterpart to the Euler scheme
(see, e.g.,\cite{s4}). The scheme for  the SDE (\ref{eq1.1}) is
given by the following recurrent formula. Set
\begin{equation}\label{eq4.2}
\left\{\begin{array}{ll}
k^n_0&=0,\qquad X^n_0=Y^n_0=Y_0,\\[2mm]
Y^n_{(k+1)/n}&=Y^n_{k/n}+f((k+1)/n,X^n_{k/n})(M_{(k+1)/n}-M_{k/n})\\[2mm]
&\qquad\quad+g((k+1)/n,X^n_{k/n})(V_{(k+1)/n}-V_{k/n}),\\[2mm]
k^n_{{(k+1)}/{n}}&= \max\big[\min\big[k^n_{{k}/{n}},
\bar u^{n}_{(k+1)/{n}}-EY^n_{(k+1)/n}\big],
\bar l^n_{(k+1)/n}-EY^n_{(k+1)/n})\big],\\[2mm]
X^n_{{(k+1)}/{n}}&= Y^n_{(k+1)/n}+k^n_{(k+1)/n},
\end{array}
\right.
\end{equation}
where $\bar l_{(k+1)/n}=H^{-1}((k+1)/n,l_{(k+1)/n},Y^n_{(k+1)/n})$,  $\bar u_{(k+1)/n}=H^{-1}((k+1)/n,u_{(k+1)/n},Y^n_{(k+1)/n})$ and
$Y^n_t=Y^n_{k/n}$, $k^n_t=k^n_{k/n}$, $X^n_t=X^n_{k/n}$ for $t\in[k/n,(k+1)/n)$, $k\in\N\cup\{0\}$, $n\in\N$.

\begin{theorem}\label{thm5}
Assume   \mbox{\rm (H)}, \mbox{\rm(A1)}, \mbox{\rm(A2)} and
\mbox{\rm(M)}.  If $X_0$ is an integrable random variable such
that $l_0\leq Eh(0,X_0)\leq u_0$,   $\{l^n\},\{u^n\}$ are
sequences of  discretizations of  $l,u$   and $(X^n, k^n)$ and
$Y^n$ are defined by \mbox{\rm(\ref{eq4.2})}, then $(X^n,
k^n)=\mathbb{SP}^{u^n}_{l^n}(h,Y^n)$, $n\in\N$  and
\begin{enumerate}\item[{\bf (i)}]
$\displaystyle{(X^n,k^n,Y^n,l^n,u^n)\arrowp (X,k,Y,l,u)\quad \,\,\mbox{
in}\,\,\Dddddd,}$
\item[{\bf (ii)}] for every $q\in\Rp$,
\[\max_{k;\,k/n\leq q}(|X^n_{k/n}-X_{k/n}|+|k^n_{k/n}-k_{k/n}|)\arrowp0,\]
\end{enumerate}
where $(X,k)$ is a unique strong solution of
\mbox{\rm(\ref{eq1.1})}.
\end{theorem}
 \begin{proof}
We know that $(X,k)=\SP{Y}$, where
$Y_t=X_0+\int_0^tf(s,X_{s-})\,dM_s+\int_0^tg(s,X_{s-})\,dV_s$,
$t\in\Rp$. Let $\hat X^n$, $\hat k^n$, $\hat Y^n$ be
discretizations of $X,k,Y$, that is $\hat X^n_t=X_{\rho^n_t}$,
$\hat k^n_t=k_{\rho^n_t}$, $\hat Y^n_t=Y_{\rho^n_t}$, $t\in\Rp$.
Let $(\bar X^n,\bar k^n)$ be a solution of the   Skorokhod problem
associated   with $\hat Y^n$, $ l^n$  and $u^n$, i.e. $(\bar X^n,
\bar k^n)=\mathbb{SP}^{u^n}_{l^n}(h,\hat Y^n)$, $n\in\N$. Then, by
Theorem \ref{thm3}(ii), for every  $q\in\Rp$,
\[
\sup_{t\leq q}|(\bar X^n_t-\hat X^n_t|+|\bar k^n_t-\hat
k^n_t|)\lra0.
\]
 Fix $q\in\Rp$. Since $E\sup_{t\leq q}|X_t|<+\infty$ and  $E\sup_{t\leq q}|Y_t| <+\infty$,  the sequences
$\{\sup_{t\leq q}|\hat X^n_t|\}$  and $\{\sup_{t\leq q}|\bar X^n_t|\}$ are uniformly integrable, which implies that
\begin{equation}\label{eq4.3}\epsilon^n_1=E \sup_{t\leq q}|\bar X^n_t-\hat
X^n_t|\lra0.\end{equation} Let $\{M^n\},\{V^n\}$ be sequences of
discretizations  of $M,V$, respectively. By (A1)   and the
Lebesgue dominated convergence theorem,
\begin{eqnarray*}\nonumber
\epsilon^n_2&=&E \sup_{t\leq q}(|\int_0^{{t}}f(s,\hat X^n_{s-})\,dM^n_s-\int_0^{\rho^n_{t}}f(s, X_{s-})\,dM_s|\\
&&\qquad\qquad+ E|\int_0^{{t}}g(s,\hat X^n_{s-})\,dV^n_s-
\int_0^{\rho^n_{t}}g(s, X_{s-})\,dV_s |)\nonumber \\
&=&E (\sup_{t\leq q}(|\int_0^{{\rho^n_{t}}}
(f(\rho^n_{s},\hat X^n_{s-})-f(s, X_{s-}))\,dM_s
+|\int_0^{{\rho^n_{t}}}(g(\rho^n_{s},\hat X^n_{s-})-
g(s, X_{s-}))\,dV_s |)\nonumber\\
&\leq&3E (|\int_0^{{\rho^n_{q}}}|f(\rho^n_{s},\hat X^n_{s-})-f(s, X_{s-})|^2\,dm_s)^{1/2} \nonumber\\&&\qquad\qquad+E\int_0^{{\rho^n_{q}}}|g(\rho^n_{s},\hat X^n_{s-})-
g(s, X_{s-})|\,dm_s |\lra0
\end{eqnarray*}
because $\rho^n_s\to s$  and $\hat X^n_{s-}\to X_{s-}$ $\Ppw$ for
$s\in[0,q]$. On the other hand, by Theorem \ref{thm3}, the pair
$(X^n, k^n)$ defined by (\ref{eq4.2})  is a solution of mean
reflected SDEs with barriers $l^n$, $u^n$ of the form
\[X^n_t=X_0+\int_0^{{t}}f(s, X^n_{s-})\,dM^n_s+\int_0^{{t}}g( s,X^n_{s-})\,dV^n_s +k^n_t,\quad t\in\Rp,\]
i.e. $(X^n, k^n)=\mathbb{SP}^{u^n}_{l^n}(h,Y^n)$, where
$Y^n_t=X_0+\int_0^{{t}}f(s, X^n_{s-})\,dM^n_s+\int_0^{{t}}g(
s,X^n_{s-})\,dV^n_s$, $t\in\Rp$, $n\in\N$. From the above and
Proposition \ref{prop1}(ii), for every $t\leq q$,
\begin{eqnarray*}\nonumber
E\sup_{s\leq t}|\hat X^n_s-X^n_s|&\leq&E\sup_{s\leq t}|\hat X^n_s-\bar X^n_s|+E\sup_{s\leq t}|\bar X^n_s-X^n_s|\\
&\leq&\epsilon^n_1+(C_h+1)E\sup_{s\leq t}|\hat Y^n_s-Y^n_s|\nonumber\\
&\leq&\epsilon^n_1+(C_h+1)\epsilon^n_2+(C_h+1)E\sup_{s\leq t}|\int_0^sf(u,\hat X^n_{u-})-f(u,X^n_{u-})\,dM^n_u|\nonumber\\
&&\quad+(C_h+1)E\sup_{s\leq t}|\int_0^sg(u,\hat X^n_{u-})-g(u,X^n_{u-})\,dV^n_s|.\nonumber\\
&=&\epsilon^n_1+(C_h+1)\epsilon^n_2+(C_h+1)E\sup_{s\leq t}|\int_0^{\rho^n_s}f(\rho^n_u,\hat X^n_{u-})-f(\rho^n_u,X^n_{u-})\,dM_u|\nonumber\\
&&\quad+(C_h+1)E\sup_{s\leq t}|\int_0^{\rho^n_s}g({\rho^n_u},\hat X^n_{u-})-g({\rho^n_u},X^n_{u-})\,dV_s|.
\end{eqnarray*}
Let $m^n$  be a discretization of the function $m$, i.e.
$m^n_t=m_{\rho^n_t}$, $t\in\Rp$, $n\in\N$. Set
$t^n_1=\inf\{t>0:(C_h+1)c\max(m^n_t,3(m^n_t)^{1/2})>1/2\}$ and
observe that arguing similarly to the proof of Theorem \ref{thm3}
shows that
\begin{eqnarray*}
E\sup_{s< t^n_1}|\hat X^n_s-X^n_s|&\leq&\epsilon^n_1+(C_h+1)\epsilon^n_2+(C_h+1)E\sup_{s< t^n_1}|\int_0^{\rho^n_s}f(\rho^n_u,\hat X^n_{u-})-f(\rho^n_u,X^n_{u-})\,dM_u|\\
&&\quad+(C_h+1)E\sup_{s< t^n_1}|\int_0^{\rho^n_s}g({\rho^n_u},\hat X^n_{u-})-g({\rho^n_u},X^n_{u-})\,dV_s|.\\
&\leq&\epsilon^n_1+(C_h+1)\epsilon^n_2+(C_h+1)3cE(\int_0^{\rho^n_{t^n_1-}}|\hat X^n_{s-}-X^n_{s-}|^2\,d\langle M\rangle_s)^{1/2}\\
&&\quad+(C_h+1)cE\int_0^{\rho^n_{t^n_1-}}|\hat X^n_{s-}-X^n_{s-}|\,d\widetilde{|V|}_s\\
&\leq& \epsilon^n_1+(C_h+1)\epsilon^n_2+(C_h+1)3c (m^n_{t^n_1-})^{1/2}E\sup_{s<t^n_1}|\hat X^n_{s-}-X^n_{s-}|\\
&&\qquad+(C_h+1)c m^n_{t^n_1-}E\sup_{s<t^n_1}|\hat X^n_{s-}-X^n_{s-}|\\
 &\leq& \epsilon^n_1+(C_h+1)\epsilon^n_2+\frac 12 E\sup_{s<t^n_1}|\hat X^n_{s}-X^n_{s}|,
\end{eqnarray*}
which implies that
\begin{equation} \label{eq4.6}
E\sup_{s<t^n_1}|\hat X^n_{s}-X^n_{s}|\leq
2(\epsilon^n_1+(C_h+1)\epsilon^n_2)\lra0.
\end{equation}
In fact the convergence holds true on the closed interval
$[0,t^n_1]$. To check this we first observe that for a
sufficiently large $n$, $t^n_1=t^*_n=\inf\{k/n;k/n\geq t_1\}\}\lra
t_1$, where $t_1$ was defined in the proof of  Theorem \ref{thm3}.
Hence $(\hat
X^n_{t^n_1-},X^n_{t^n_1-},Y^n_{t^n_1-})\lra(X_{t_1-},X_{t_1-},Y_{t_1-})$.
Moreover,
\[
k^n_{t^n_1}=\max(\min(k^n_{t^n_1-},\bar
u^n_{t^n_1}-EY^n_{t^n_1}),\bar l^n_{t^n_1}-EY^n_{t^n_1})),\] where
$Y^n_{t^n_1}=Y^n_{t^n_1-}+f(t^n_1,X^n_{t^n_1-})\Delta
V^n_{t^n_1}+g(t^n_1,X^n_{t^n_1-})\Delta M^n_{t^n_1}$, $\bar
l^n_{t^n_1}=H^{-1}(t^n_1,l^n_{t^n_1},Y^n_{t^n_1})$  and $\bar
u^n_{t^n_1}=H^{-1}(t^n_1,u^n_{t^n_1},Y^n_{t^n_1})$, which implies
that $\Delta X^n_{t^n_1}=\Delta Y^n_{t^n_1}+\Delta
k^n_{t^n_1}\arrowp\Delta X_{t_1}$. On the other hand,  $\Delta
\hat X^n_{t^n_1}\lra\Delta X_{t_1}$ $\Ppw$, which  implies that
$\hat X^n_{t^n_1}-\Delta X^n_{t^n_1}\arrowp0$. Therefore, by
(\ref{eq4.6}) and uniform integrability of jumps $\{\Delta \hat
X^n_{t^n_1}\}$, $\{\Delta X^n_{t^n_1}\}$ we have $ E\sup_{s\leq
t^n_1}|\hat X^n_{s}-X^n_{s}|\lra0$.

 It is easy to see that we can
repeat the previously used arguments on next intervals
$[t^n_k,t^n_{k+1}]$, where
$t^n_{k+1}=t^n_k+\inf\{t>0:(C_h+1)c\max(\hat m^n_t,3(\hat
m^n_t)^{1/2})>1/2\}$ and $\hat m^n_t=m^n_{t^n_k+t}-m^n_{t^n_1}$,
$k\in\N$. Since for a sufficiently large $n$,
$t^n_k=\inf\{k/n:k/n\geq t_k\}\}$ ($t_k$ was defined in the proof
of  Theorem  \ref{thm3}), for every $q\in\Rp$ in finitely many
steps we are able to prove  that
\begin{equation}
\label{eq4.7} E\sup_{s\leq q}|\hat
X^n_{s}-X^n_{s}|\lra0,\quad q\in\Rp.
\end{equation}
From (\ref{eq4.7})  and  the observation that $(\hat
X^n,M^n,V^n,l^n,u^n)\to (X,M,V,l,u)$ $\Ppw$ in $\Dddddd$ the
theorem follows.
\end{proof}

We say that equation \eqref{eq1.1} has a {weak
solution} if there exist a filtered probability space
$(\bar{\Omega},\bar{\FF},(\bar{\cal
F}_t),\bar{P})$,  adapted processes $(\bar X,\bar k)$,  $\bar M,\bar V$ defined on
$(\bar{\Omega},\bar{\FF},\bar{P})$ such that
${\cal L}(\bar{X}_0,\bar M,\bar V)=$ ${\cal L}({X}_0, M, V)$ and $(\bar X,\bar k)$ is a solution of mean reflected SDE
\begin{equation}\label{eq4.8}
\bar X_t = \bar X_0 +\int_0^t f(s,\bar X_{s-})\,d\bar M_s+\int_0^tg(s,\bar X_{s-})\, d\bar V_s + \bar k_t,\quad t\in\Rp.
\end{equation}
If the laws ${\cal L}(\bar{X},\bar k,\bar M,\bar V)$ and ${\cal
L}(\bar{X}',\bar k',\bar M',\bar V')$ of any two weak solutions of
\eqref{eq1.1}, possibly defined on different probability spaces,
are the same  we say that the weak uniqueness holds for
\eqref{eq1.1}.

 Let $\{M^n\}$ and $\{V^n\}$ be  sequences of square integrable $({\cal
F}^n_t)$-martingales and $({\cal F}^n_t)$-adapted process with
trajectories in $\D$ such that $E|V^n|_q<\infty$ for every
$q\in\Rp$. We will consider a sequence of mean reflected SDEs of
the form (\ref{eq4.9}). We assume that the  starting points
$X^n_0$ are integrable ${\cal F}^n_0$-measurable random variables
and for every $n\in\N$ the processes $M^n$, $V^n$ satisfy the
condition
\begin{enumerate}
\item[(Mn)] there exists a nondecreasing c\`adl\`ag function
$m^n:\Rp\rightarrow\R$ with $m^n_0=0$ such that
\[\max(\langle M^n\rangle_t,\widetilde{|V^n|}_t)\le m^n_t,\quad t\in\Rp.\]
\end{enumerate}
Clearly, if $\sup_n m^n_t<+\infty$, $t\in\Rp$, then  the sequences
$\{M^n\}$ and $\{V^n\}$ satisfy the so called condition (UT)
introduced by Stricker\cite{st}:
\begin{description}
\item[{(UT)}]\index{Condition (UT)} For every  $q\in\Rp$ the
family of random variables
\[
\{\int_{[0,q]} U^n_s\,dZ^n_s;\,n\in\N\,,\,U^n\in\mbox{\bf U}^n_q\}
\]
is bounded in probability. Here $ \mbox{\bf U}^n_q$  is the class
of discrete predictable processes of the form
$U^n_s=U^n_0+\sum_{i=0}^kU^n_i\mbox{\bf 1}_{\{t_i<s\leq
t_{i+1}\}}$, where  $0=t_0<t_1<...<t_k=q$  and $U^n_i$  is  ${\cal
F}^n_{t_i}$  measurable, $|U^n_i|\leq 1$  for
$i\in\{0,...,k\}\,,\,n,k\in\N.$
\end{description}
Condition (UT)  proved to be very useful in the theory of limit
theorems for stochastic integrals  and for solutions of SDEs (see,
e.g., \cite{jmp, kp, ms,s4, Sl1}).

We will also assume that
\begin{enumerate}
\item[(A3)] the coefficients $f^n$, $g^n$ satisfy (A1) for every $n\in\N$ and there exists $f,g$ such that
\[||f^n-f||_{[0,q]\times K}+||g^n-g||_{[0,q]\times K}\lra0\]
for every $q\in\Rp$  and every compact subset $K\subset \R$.
\end{enumerate}

\begin{theorem}\label{thm6}
Assume  \mbox{\rm(H)}, \mbox{\rm(A1)}  and \mbox{\rm(A3)}. Let
$\{l^n\},\{u^n\}$ be sequences of  c\`adl\`ag functions such that
$l^n\leq u^n$, $\{M^n\}$, $\{V^n\}$ be sequences of processes
satisfying \mbox{\rm(Mn)} and $\{(X^n,k^n)\}$ be a sequence of
solutions of $\mbox{\rm(\ref{eq4.9})}$, $n\in\N$.  If $\sup_n
m^n_t<+\infty$, $t\in\Rp$,  $\sup_n E|X^n_0|<+\infty$, $l^n_0\leq
Eh(0,X^n_0)\leq u^n_0$, $n\in\N$, and
\begin{equation}\label{eq4.10}
(X^n_0,M^n,V^n,v^n,l^n,u^n)\arrowd (X_0,M,V,v,l,u) \quad in
\,\,\R\times\Dddddd,\end{equation} where $v^n_t=EV^n_t$,
$v_t=EV_t$, $t\in\Rp$, then
\begin{equation}\label{eq4.15}\{(X^n,k^n,M^n, V^n,v^n,l^n,u^n)\}\quad is
\,\,tight\,\,in\,\,\Dddddddd
\end{equation} and its each limit
point is a weak solution of the  mean reflected SDE
\mbox{\rm(\ref{eq1.1})}.
\end{theorem}
\begin{proof}
First we show that
\begin{equation}
\label{eq4.11} \sup_nE\sup_{t\leq q}|X^n_t|,\quad q\in\Rp.
\end{equation}
Fix $q\in\Rp$ and set
$t^n_1=\inf\{t>0:(C_h+1)\mu\max(m^n_t,3(m^n_t)^{1/2})>1/2\}\wedge
q$. By (\ref{eqcor1}), for every $t\leq q$,
\begin{eqnarray*}
E\sup_{s< t}|X^n_s|&\leq&E|X^n_0| +(C_h+1)E\sup_{s< t}|Y^n_s-Y^n_0  |\\
&&\qquad+C_h(\lambda_ht+
\sup_{s<t}\max(|l^n_s-l^n_0|,|u^n_s-u^n_0|)\big),
\end{eqnarray*}
where $Y^n=X^n_0+\int_0^\cdot f^n(u,X^n_{u-})\,d M^n_u+\int_0^\cdot
g^n(u,X^n_{u-})\, dV^n_u$, $n\in\N$. Since $\sup_nE|X^n_0|<+\infty$
and $\sup_n\sup_{t\leq q}\max(|l^n_s|,|u^n_s|)<+\infty$, it
follows from  (Mn) and  (A1) that
\begin{eqnarray*}E\sup_{s< t^n_1}|X^n_s|&\leq&C(h,q) +(C_h+1)\max(3(m^n_{t^n_1-})^{1/2},m^n_{t^n_1-})\mu(1+E\sup_{s<t^n_1}|X^n_s|)\\
&\leq& C(h,q) +(C_h+1)\max(3(m^n_{t^n_1-})^{1/2},m^n_{t^n_1-})\mu
+\frac12E\sup_{s<t^n_1}|X^n_s|,
\end{eqnarray*}
and hence that
\[
\sup_nE\sup_{s< t^n_1}|X^n_s|<+\infty.
\] Since  $\Delta
Y^n_{t^n_1}=f^n(t^n_1,X^n_{t^n_1-})\Delta
M^n_{t^n_1}+g(t^n_1,X^n_{t^n_1}-)\Delta V^n_{t^n_1}$, this implies
that   $\sup_nE|\Delta Y^n_{t^n_1}|<+\infty$. Hence also
$\sup_n|\Delta k^n_{t^n_1}|<+\infty$. Consequently,
$\sup_nE\sup_{s\leq t^n_1}|X^n_s|<+\infty$.

Similarly to the proof of Theorem \ref{thm4}  we can repeat the
previously used arguments on the next intervals
$[t^n_k,t^n_{k+1}]$, where
$t^n_{k+1}=t^n_k+\inf\{t>0:(C_h+1)\mu\max(\hat m^n_t,3(\hat
m^n_t)^{1/2})>1/2\}$ and $\hat m^n_t=m^n_{t^n_k+t}-m^n_{t^n_1}$,
$k\in\N$. What is left is to show that
$\sup_nj(n)=\inf\{k:t^n_k=q\}<+\infty$. To  check this set
$r^n(t)= \max(m^n_t,3(m^n_t)^{1/2})$, $r=\sup_n r^n(q)$ and
observe that $r^n_0=0$, $r^n$ is nondecreasing  and  for every
$k$, $\frac12\le (C_h+1)\mu(r(t^n_{k+1})-r(t^n_{k}))$, which
implies that for every $n\in\N$,
\[j(n)\frac12\le(C_h+1)\mu r^n(q)\leq(C_h+1)\mu r,\]
which  completes the proof of (\ref{eq4.11}).

In next step we show (\ref{eq4.15}). By  (\ref{eq4.11}) and (A1),
for every $q\in\Rp$ the sequences $\{\sup_{t\leq
q}|f^n(t,X^n_t)|\}$  and $\{\sup_{t\leq q}|g^n(t,X^n_t)|\}$ are
bounded in probability. Therefore, by (Mn) and \cite[Lemma
1.6]{ms}, $\{Y^n\}$ satisfies (UT). Moreover, for every
$\epsilon>0$,
\begin{equation}\label{eq4.13}
\{\bar V_{2+\epsilon}(Y^{n})_q\}\,\,\mbox{\rm is bounded in
probability.} \end{equation} Indeed, if we set
$\tau^N_n=\inf\{t:\min(|f^n(t,X^n_t)|,|g^n(t,X^n_t)|>N)\}$ for
$n,N\in\N$, then
\[
\lim_{N\to\infty}\limsup_{n\to \infty}P(\tau_n^N<q)=0,\quad
q\in\Rp,
\] and by
the estimate from \cite{kubilius:2009} (see also \cite[Remark
2.4]{FS}),
\begin{eqnarray*}
P(\bar V_{2+\epsilon}(\int_0^\cdot f^n(s,X^{n}_{s-})dM^n_s)_q>K)&\leq& P(\bar V_{2+\epsilon}(\int_0^\cdot f^n(s,X^{n}_{s-})dM^{n,\tau_n^N}_s)_q>K)+P(\tau_n^N<q)\\
&\leq&\frac{E\bar V_{2+\epsilon}(\int_0^\cdot f^n(s,X^{n}_{s-})dM^{n,\tau_n^N}_s)_q}{K}+P(\tau_n^N<q)\\
&\leq&\frac{C(\epsilon)E(\int_0^q |f^n(s,X^{n}_{s-})|^2d\langle M^{n,\tau_n^N}\rangle_s)}{K}+P(\tau_n^N<q)\\
&\leq&\frac{C(\epsilon)N^2 m^{n}_q}{K}+P(\tau_n^N<q),
\end{eqnarray*}
which implies that  $\{\bar V_{2+\epsilon}(\int_0^\cdot
f^n(s,X^{n}_{s-})dM^n_s)_q\}$ is bounded in probability. Since the
similar property for integrals driven by processes with uniformly
bounded variation is straightforward, the proof of (\ref{eq4.13})
is completed.   Unfortunately, the sequence of solutions $\{(X^n,
k^n)=\mathbb{SP}^{u^n}_{l^n}(h,Y^n)\}$ need not satisfy (UT) and
it is not even clear if is possible to approximate them by
sequences satisfying (UT). Consequently, standard tightness
criterions from \cite{jmp, kp, ms} in the case of solutions of
\eqref{eq4.9} are useless.  In the present paper, we use new
tightness results recently proved in \cite[Proposition 4.1]{FS}.
We will show that it is possible to approximate solutions of
\eqref{eq4.9} by processes having uniformly bounded in probability
$2+\epsilon$\,-variation for $\epsilon>0$.

Set $\gamma ^{i}_0=0$, $\gamma ^{i}_{k+1}=
\min(\gamma^{i}_{k}+\delta^i_k, \inf \{ t > \gamma ^{i}_{k}:
\min(|\Delta l_t |, |\Delta u_t|)> \delta^i \})$ and $\gamma
^{ni}_0=0$, $\gamma ^{ni}_{k+1}= \min(\gamma ^{ni}_{k}+\delta^i_k,
\inf \{ t > \gamma ^{ni}_{k}: \min(|\Delta l^n_t |, |\Delta
u^n_t|)> \delta^i \})$, where $\{ \delta^i \}$, $\{ \{ \delta^i_k
\} \}$ are families of positive constants such that $\delta^i
\downarrow 0$, $\delta^i/2 \leq \delta^i_k \leq \delta^i$ $|\Delta
l_t|\neq\delta^i$, $|\Delta u_t|\neq\delta^i$, $t \in \Rp$,
$\Delta l_{\gamma_k^{i}+\delta^i_k} =0$, $\Delta
u_{\gamma_k^{i}+\delta^i_k} =0$. For every $i\in\N$ define new
sequences  $\{l^{n,(i)}\}$, $\{u^{n,(i)}\}$ of functions by
putting $l^{n,(i)}_t=l^n_{\gamma ^{ni}_{k}}$,
$u^{n,(i)}_t=u^n_{\gamma ^{ni}_{k}}$, $t \in [\gamma ^{ni}_{k},
\gamma ^{ni}_{k+1} )$, $k \in \No $, $n\in \N$,
$l^{(i)}_t=l_{\gamma ^{i}_{k}}$, $u^{(i)}_t=u_{\gamma ^{i}_{k}}$,
$t \in [\gamma ^{i}_{k}, \gamma ^{i}_{k+1} )$, $k \in \No $. Then
using the continuous mapping argument, we have
$(l^n,l^{n,(i)},u^n,u^{n,(i)}) \rightarrow (l,l^{(i)},u,u^{(i)})$
in $\DDDD$, which implies that
\begin{equation} \label{ni1} \lim_{i \rightarrow
\infty }\limsup_{n \rightarrow \infty }
 \sup_{t \leq q}\max(|l^{n,(i)}_t-l^n_t|,|u^{n,(i)}_t-u^n_t|) \lra0, \quad  q \in \Rp.
 \end{equation}
 Moreover, for every $i\in\N$,
\begin{equation} \label{ni2}
 \sup_n\max(|l^{n,(i)}|_q, |u^{n,(i)}|_q)<+\infty, \quad  q \in \Rp.
 \end{equation}
Set $\{(X^{n,(i)},k^{n,(i)})
=\mathbb{SP}^{u^{n,(i)}}_{l^{n,(i)}}(h,Y^n)\}$, $n,i\in\N$, and
observe that by (\ref{ni1})  and Proposition  \ref{prop1},
\begin{equation}\label{eq4.12} \lim_{i \rightarrow \infty }\limsup_{n
\rightarrow \infty } E\sup_{t \leq q}|X^{n,(i)}_t-X^n_t| =0,\quad
  q\in\Rp.\end{equation}
On the other hand, by simple  calculations,  for all $q\in\Rp$,
$i\in\N$ there exists $C(h,q,i)>0$ such that for every $n\in\N$,
\begin{eqnarray*}
v_2(k^{n,(i)})_q&=&\sup_\pi \sum_{j=1}^m
|k_{t_j}-k_{t_{i-j}}|^2\\&\leq& 45 C_h^2\mu^2(1+E\sup_{t\leq
q}|X^n_t|)m_q+5 C_h^2\mu^2(1+E\sup_{t\leq
q}|X^n_t|)(m_q)^2\\
&&\qquad+5C_h^2\lambda_h^2(q^2+|l^{n,(i)}|_q^2+|u^{n,(i)}|_q^2)\leq
C(h,q,i).
\end{eqnarray*}
Since
 \[\bar V_{2+\epsilon}(X^{n,(i)})_q\leq \bar V_{2+\epsilon}(Y^{n})_q+(C(h,q,i)^{1/2},\]
 it follows  from (\ref{eq4.13}) that  also the sequence
 $\{\bar V_{2+\epsilon}(X^{n,(i)})_q\}$ is
bounded in probability.
 Furthermore,  it is well
known that for continuous $f: \Rp\times\R \lra \R $  one can
construct  a sequence $\{ f^{(i)} \}$  of  functions such that
$f^{(i)} $ satisfies (A2), for $i\in\N$  and  $||f^{(i)} -
f||_{[0,q]\times K} \lra 0$
 for any compact subset $K \subset \R$ and $q\in\Rp$. Similarly
one can construct a sequence of   functions $\{ g^{(i)} \}$
 satisfying  (A2) for  $i\in\N$ and approximating $g$.
Since $f^{(i)},g^{(i)}  $ satisfy (A2), $\{X^{n,(i)}\}$ has
bounded $2+\epsilon$\,-variation and $\{ f^{(i)}(X^{n,(i)}) \}$,
$\{ g^{(i)}(X^{n,(i)}) \}$ have bounded $2+\epsilon$\,-variation
as well. Using \cite[Proposition 4.1]{FS}  and the fact that the
sequences $\{M^n\}$ and $\{V^n\}$ satisfy  (UT)  shows that
\[
\{(X^n_0,M^n, \int_0^\cdot
f^{(i)}(s,X^{n,(i)}_{s-})\,dM^n_s, V^n,\int_0^\cdot
g^{(i)}(s,X^{n,(i)}_{s-})\,dV^n_s,v^n,l^n,u^n)\}
\]
 is tight in  $\R\times\Dddddddd$. Therefore, by (\ref{eq4.12}), also
\[\{(X^n_0,M^n, \int_0^\cdot
f^n(s,X^{n}_{s-})\,dM^n_s, V^n,\int_0^\cdot
g^n(s,X^{n}_{s-})\,dV^n_s,v^n,l^n,u^n)\}
\]
is tight in $\R\times\Dddddddd$,
which together with Theorem \ref{thm2} implies (\ref{eq4.15}).

Assume that there exists a subsequence $(n') \subset (n )$ such
that
\begin{equation}\label{eq4.14}
(X^{n'},k^{n'},M^{n'}, V^{n'},v^{n'},l^{n'},u^{n'})
\arrowd(\bar X,\bar k,\bar M,\bar V,v,l,u)
\quad\,\,\mbox{\rm in}\,\,\Dddddddd,
\end{equation}
where ${\cal L}(\bar X^n_0,\bar M,\bar V)={\cal L}(\bar X_0,M,V)$.
Then, by \cite[Proposition 4.1(ii)]{FS},
\[(X^{n'},k^{n'},M^{n'}, V^{n'},Y^{n'},l^{n'},u^{n'})\arrowd(\bar X,\bar k,\bar M,\bar V,\bar Y,l,u)\quad\,\,\mbox{\rm in}\,\,\Ddddddddd.\]
Now set $y^{n'}_t=EY^{n'}_t=E\int_0^t
g(s,X^{n'}_{s-})\,dV^{n'}_s$, $\bar y_t=E\bar Y_t=E\int_0^t
g(s,\bar X_{s-})\,d\bar V_s$, $t\in\Rp$. Since $\{t:\Delta \bar
y_t\neq0\}\subset\{t:\Delta v_t\neq0\}$, from the above
convergence and (\ref{eq4.10}) we deduce that
\begin{equation}
\label{4.16}(X^{n'},k^{n'},M^{n'},
V^{n'},Y^{n'},y^{n'},l^{n'},u^{n'}) \arrowd(\bar X,\bar k,\bar
M,\bar V,\bar Y,\bar y,l,u)\quad\,\, \mbox{\rm in}\,\,\Ddddddddd.
\end{equation}
Therefore from Theorem \ref{thm3}(i) it follows that $(\bar X,
\bar k)=\mathbb{SP}^{u}_{l}(h,\bar Y)\}$ with $\bar Y$ defined as
$\bar Y = \bar X_0 + \int_0^{\cdot}\;f(s,\bar X_{s-})\,d\bar
M_s+\int_0^{\cdot}\;g(s,\bar X_{s-})\,d\bar V_s$, which implies
that $(\bar X,\bar k)$ is a  weak solution of the SDE
(\ref{eq1.1}).
\end{proof}

In the proof of existence of weak  solutions  of (\ref{eq1.1}) we
will use the approximation  scheme defined in (\ref{eq4.2}).

\begin{corollary}\label{cor4.1}Assume \mbox{\rm(H)},
\mbox{\rm(A1)} and \mbox{\rm (M)}. If $l_0\leq Eh(0,X_0)\leq u_0$,
then there exists  a weak solution $(X,k)$ of
\eqref{eq1.1}.
\end{corollary}
\begin{proof}
Let $\{M^n\},\{V^n\}$ and $\{l^n\},\{u^n\}$ be sequences of
discretizations  of $M,V$  and $l,u$.  Clearly
\begin{equation}\label{eq411}(M^n,V^n,v^n,l^n,u^n)
\lra (M,V,v,l,u)\quad \Ppw\,\,\mbox{\rm
in}\,\,\DDDDD.
\end{equation}  If $(X^n, k^n)$ and
$Y^n$ are defined by \eqref{eq4.2}, then $(X^n,
k^n)=\SPp{Y^n}{l^n}{u^n}$, where $Y^n_t=X_0+\int_0^{{t}}f(s,
X^n_{s-})\,dM^n_s+\int_0^{{t}}g(s,X^n_{s-})\,dV^n_s$, $t\in\Rp$.
This means that $(X^n, k^n)$, $n\in\N$, are solutions  of
(\ref{eq4.9}) with $f^n=f$  and $g^n=g$. Now observe that using
(M) we are able to estimate the predictable characteristics of
$M^n$ and $V^n$. Indeed, for every $t\in\Rp$,
\begin{eqnarray*}
\langle M^n\rangle_t&=&\sum_{k;k/n\leq t}E(M_{k/n}-M_{(k-1)/n})^2|{\cal F}_{(k-1)/n})\\
 &=&\sum_{k;k/n\leq t}E(\langle M\rangle^{k/n}_{(k-1)/n}|{\cal F}_{(k-1)/n})\leq m^n_t
 \end{eqnarray*} and\begin{eqnarray*}
\widetilde{|V^n|}_t &=&\sum_{k;k/n\leq t}E(|V_{k/n}-V_{(k-1)/n}||{\cal F}_{(k-1)/n})\\
 &=&\sum_{k;k/n\leq t}E(\widetilde{|V|}^{k/n}_{(k-1)/n}|{\cal F}_{(k-1)/n})\leq m^n_t,
 \end{eqnarray*}
where $m^n$ is a discretization of the function $m$, i.e.
$m^n_t=m_{\rho^n_t}$, $t\in\Rp$. Consequently, the condition (Mn)
is satisfied  and the existence of a weak solution  follows from
Theorem \ref{thm6}.
\end{proof}

\begin{corollary}\label{cor4.2} Under the assumptions of
Theorem \ref{thm6}, if moreover \mbox{\rm(A2)} is satisfied and
the processes $M,V$ have independent increments,
\[
(X^{n},k^{n},M^{n}, V^{n},l^{n},u^{n})\arrowd(X,k,M,V,l,u)
\quad\,\,\mbox{\rm in}\,\,\Ddddddd,
\]
where $(X,k)$ is the unique weak solution of
\mbox{\rm(\ref{eq1.1})}.
\end{corollary}
\begin{proof}
It is sufficient to apply Theorem  \ref{thm6} and observe that in
the case where  \mbox{\rm(A2)} is satisfied and $M,V$ have
independent increments then   weak uniqueness for (\ref{eq1.1})
holds true.  To check this assume that $(\bar{X},\bar k,\bar
M,\bar V)$ and $(\bar{X}',\bar k',\bar M',\bar V')$  are two weak
solutions of (\ref{eq1.1})  and  ${\cal L}(\bar{X}_0,\bar M,\bar
V)=$ ${\cal L}(\bar{X}'_0,\bar{ M}', \bar{V}')$ $={\cal
L}({X}_0,M, V)$.  Since $M,V$ have independent increments, the
processes $\bar M,\bar V$  and $\bar{M}',\bar{V}'$ also have
independent increments.  Moreover, they have the same
deterministic predictable  characteristics $\langle
M\rangle,\widetilde{|V|}$. Let $(\bar X^n,
k^n)=\mathbb{SP}^{u^n}_{l^n}(h,\bar Y^n)$, $n\in\N$, and
$(\bar{X^n}', k^n)=\mathbb{SP}^{u^n}_{l^n}(h,\bar{Y^n}')$,
$n\in\N$, be the approximations of $(\bar{X},\bar k)$ and
$(\bar{X}',\bar k')$, respectively, considered in Theorem
\ref{thm5}.  By Theorem \ref{thm5},
$
(\bar X^n,k^n,\bar Y^n,l^n,u^n)\arrowp (\bar X,\bar k,\bar Y,l,u)$ in $\Dddddd$
and
$(\bar{ X^n}',k^n,\bar {Y^n}',l^n,u^n)\arrowp (\bar{ X}',\bar k',\bar Y,l,u)$ in $\Dddddd$.
Clearly, the  equality of laws ${\cal L}(\bar{X}_0,\bar M,\bar
V)=$ ${\cal L}(\bar{X}'_0,\bar{ M}', \bar{V}')$ implies the
equality of laws of the approximations, i.e. ${\cal L}(\bar
X^n,k^n,\bar Y^n,l^n,u^n)={\cal L}(\bar {X^n}',k^n,\bar
{Y^n}',l^n,u^n)$, $n\in\N$, which completes the proof.
\end{proof}

\begin{remark}
{\rm From  Proposition \ref{prop4} one can deduce that in case the
processes $\{V^n\}$ have independent increments or $P(\Delta
V_t=0)=1$, $t\in\Rp$, in Theorem \ref{thm6}  and   Corollary
\ref{cor4.2} in place of (\ref{eq4.10})  it suffices to assume
that
\[
(X^n_0,M^n,V^n,l^n,u^n)\lra (X_0,M,V,l,u) \quad \mbox{\rm in} \,\,\R\times\Ddddd.\]
}\end{remark}

\section{Appendix.  Skorokhod problem, deterministic case}

We begin with a general definition of  the  Skorokhod problem with
time-dependent reflecting barriers for c\`adl\`ag functions
considered in \cite{BKR}.
\begin{definition}{\rm  Let $y,l,u\in\D$ with
\label{def5} $l\leq u$ and $l_0\leq y_0\leq u_0$. We say that
a~pair $(x,k)\in\DD$ with $k_0=0$ is a solution of the  Skorokhod
problem associated with $y$ and  barriers $l,u$ (and we write $(x,
k)=SP_l^u(y)$)  if
\begin{description}
\item[(i)]$ x_t = y_t+k_t\in[l_t,u_t]$, $t\in\Rp$,
\item[(ii)]for every $0\leq t\leq q$,
\begin{eqnarray*}
k_q-k_t\geq0,&&\quad\mbox{\rm if}\,\,x_s<u_s\,\,\mbox{\rm for
all}\,\,s\in(t,q],\\
k_q-k_t\leq0,&&\quad\mbox{\rm if}\,\,x_s>l_s\,\,\mbox{\rm for
all}\,\,s\in(t,q],\end{eqnarray*}and for every $t\in\Rp$, $ \Delta
k_t\geq0$ if $x_t<u_t$ and $\Delta k_t\leq0$ if $x_t>l_t$.
\end{description}
}
\end{definition}
\begin{theorem}(\cite[Theorem 2.6]{BKR}) \label{thm5.1}
There exists a unique solution  of the  Skorokhod problem
$(x,k)=SP_l^u(y)$. Moreover,
\begin{equation}\label{eq5.1}
k_t=-\max(0\wedge\inf_{0\leq u\leq t}(y_u-l_u),
\sup_{0\leq s\leq t}[(y_s-u_s)\wedge\inf_{s\leq u\leq t}(y_u-l_u)]),
\quad t\in\Rp.
\end{equation}
\end{theorem}
One can  also show Lipschitz continuity  of the mapping
$(y,l,u)\mapsto(x,k)$ in supremum norm.
\begin{theorem}(\cite[Theorem 2.6]{sl-wo/13}) \label{thm5.2} If
$(x^i,k^i)$ is a solution associated with ${y^i}\in\D$ and
barriers $l^i,u^i$, $i=1,2$, then for every $q\in\Rp$,
\begin{equation}\label{eq5.2}\sup_{t\leq
q}|x^1_t-x^2_t|\leq2\sup_{t\leq q}|y^1_t-y^2_t| +\sup_{t\leq
q}\max(|l^1_t-l^2_t|,|u^1_t-u^2_t|)\end{equation} and
\begin{equation}\label{eq5.3}
 \sup_{t\leq q}|k^1_t-k^2_t|\leq\sup_{t\leq q}|y^1_t-y^2_t|
+\sup_{t\leq q}\max(|l^1_t-l^2_t|,|u^1_t-u^2_t|).\end{equation}
\end{theorem}

From (\ref{eq5.2}), (\ref{eq5.3}) it is  easy to deduce stability
of the mapping $(y,l,u)\mapsto(x,k)$ in the Skorokhod topology
$J_1$. More precisely, if $(y^n,l^n,u^n)\lra(y,l,u)$ in $\DDD$,
then
\begin{equation}\label{eq5.4}
 (x^n,k^n,y^n,l^n,u^n)\lra(x,k,y,l,u)\quad in\,\, \DDDDD.\end{equation}
It is worth adding that in   \cite[Definition 2.1]{sl-wo/13}
condition (ii)   is replaced by
\begin{itemize}
\item for every $0\leq t\leq q$ such that
$\inf_{s\in[t,q]}(u_s-l_s)>0$ the function $k$ has bounded
variation on $[t,q]$  and
\[\int_{[t,q]}(x_s-l_s)\,dk_s\leq0\quad\mbox{\rm
and}\quad\int_{[t,q]}(x_s-u_s)\,dk_s\leq0.\]
\end{itemize}
However, simple calculations shows that in fact these two
definitions are equivalent.

 In the classical Skorokhod problem it is assumed that
the function $k$ has bounded variation on  any  bounded interval
$[0,q]$, or, equivalently,  $k=k^{(+)}-k^{(-)}$,
 where $k^{(+)}$, $k^{(-)}$ are nondecreasing right continuous
functions with $k_0=k^{(+)}_0=k^{(-)}_0=0$ such that $k^{(+)}$
increases only on $\{t:x_t={l}_t\}$ and  $k^{(-)}$ increases only
on $\{t:x_t={u}_t\}$. In \cite{rutkowski:1980,ChL} and \cite{DN}
it is observed that the above property holds true in the case
where the barriers $l,u\in\D$ satisfy the additional  condition
\begin{equation}\label{eq5.5}
\inf_{t\leq q}({u}_t-{l}_t)>0,\quad q\in\Rp.
\end{equation}
Under (\ref{eq5.5}) it is possible to estimate the variation of
the function $k$ compensating the reflections  with the use of
$\eta$-oscillations of $y,l,u$.

\begin{proposition}(\cite[Proposition 2.11]{sl-wo/10})
\label{prop5}For any $q\in\Rp$ and $\eta$ such that
$0<2\eta\leq\inf_{t\leq q}({u}_t-{l}_t)/3$ we have
\begin{equation}\label{eq5.6}|k|_q\leq
6(N_{\eta}(y,q)+N_{\eta}(l,q)+N_{\eta}(u,q)+1)(\sup_{ t\leq q}|y_t|+\sup_{ t\leq
q}\max(|l_t|,|u_t|)).\end{equation}
\end{proposition}

We close this section with simple remark showing that in the
special case of the Skorokhod problems with one barrier, i.e. when
$u=+\infty$ or $l=-\infty$, the function $k$ has much simpler
form. Indeed, in  case of one lower barrier $l$ one can check that
$k$ is a nondecreasing function of the form
\begin{equation}\label{eq5.7}
k_t=\sup_{s\leq t}(y_s-l_s)^-=\sup_{s\leq t}(l_s-y_s)^+,\quad
t\in\Rp.
\end{equation}
Similarly, in case of one upper barrier $u $ one can check that
$k$  is a nonincreasing function of the form
\begin{equation}\label{eq5.8}
k_t=-\sup_{s\leq t}(y_s-u_s)^+=-\sup_{s\leq t}(u_s-y_s)^-,\quad
t\in\Rp.\end{equation}

\noindent{\bf Acknowledgements}\\
{This work was supported by the Polish National Science Centre
under Grant  \\ 2016/23/B/ST1/01543).}

 \bibliographystyle{plain}

\end{document}